\documentclass[a4paper]{article}
\usepackage[top=1in, bottom=1in, left=1.1in, right=1.8in]{geometry}
\usepackage{amsfonts}
\usepackage{amsmath}
\usepackage{amsthm,amssymb}
\usepackage{CJK,graphicx}
\usepackage{amscd}
\usepackage{amssymb}
\usepackage{mathrsfs}
\usepackage[all,cmtip]{xy}
\usepackage{lmodern}
\usepackage[Symbol]{upgreek}
\usepackage{bm}
\usepackage{eucal}
\usepackage[nopar]{lipsum}
\newtheorem{theorem}{Theorem}[section]
\newtheorem{corollary}{Corollary}[section]
\newtheorem{definition}{Definition}[section]
\newtheorem{lemma}{Lemma}[section]
\newtheorem{proposition}{Proposition}[section]
\newtheorem{assumption}{Assumption}[section]
\newtheorem{example}{Example}[section]

\AtEndDocument{\bigskip{\footnotesize
  \textsc{The Institute of Mathematical Sciences and Department of Mathematics, The Chinese University of Hong Kong, Shatin, Hong Kong} \par
  \textit{E-mail address}: \texttt{leung@math.cuhk.edu.hk} \par
  \addvspace{\medskipamount}
  \textsc{King's College London} \par
  \textit{E-mail address}: \texttt{yin.li.16@ucl.ac.uk} \par
}}
\begin{document}
\title{\textbf{Twin Lagrangian fibrations in mirror symmetry}}\author{Naichung Conan Leung and Yin Li}\date{}\maketitle
\begin{abstract}
A twin Lagrangian fibration, originally introduced by Yau and the first author, is roughly a geometric structure consisting of two Lagrangian fibrations whose fibers intersect with each other cleanly. In this paper, we show the existence of twin Lagrangian fibrations on certain symplectic manifolds whose mirrors are fibered by rigid analytic cycles. Using family Floer theory in the sense of Fukaya and Abouzaid, these twin Lagrangian fibrations are shown to be induced from fibrations by rigid analytic subvarieties on the mirror. As additional evidences, we discuss two simple applications of our constructions.
\end{abstract}

\section{Introduction}

\subsection{Background}
This paper is devoted to the study of certain geometric structures arising naturally from the context of mirror symmetry. The Strominger-Yau-Zaslow conjecture $\cite{syz}$ asserts that mirror symmetry should be understood in a geometric way via the so-called $T$-\textit{duality}. Roughly speaking, this means that the mirror $Y^\vee$ of a Calabi-Yau manifold $Y$ should be interpreted as the moduli space of the Lagrangian branes $(F_b,\xi_b)$, where $F_b$ is a fiber of the (special) Lagrangian torus fibration $\pi:Y\rightarrow B$ and $\xi_b$ a unitary rank 1 local system. We call $\pi$ the \textit{SYZ fibration} on $Y$.\\
In $\cite{ly}$, Yau and the first author applied such an interpretation of mirror symmetry to the case when $Y^\vee$ carries an additional elliptic fibration structure, and conjectured that there should be another Lagrangian fibration $\pi_\star:Y\rightarrow B_\star$ which is compatible with the original SYZ fibration $\pi$ in certain sense. They call such a geometric structure
\begin{equation}B\xleftarrow{\pi}Y\xrightarrow{\pi_\star}B_\star\end{equation}
a \textit{twin Lagrangian fibration} on $Y$, see Definition \ref{definition:4.1} below. It is not hard to see the heuristic arguments used in $\cite{ly}$ can be extended to the more general case when $Y^\vee$ admits a fibration by rigid analytic subvarieties.\bigskip

On the other hand, the \textit{family Floer theory} developed by Fukaya and Abouzaid in $\cite{kf,ma,ma1}$ allows us to see precisely what is the mirror object of a tautologically unobstructed Lagrangian submanifold $L\subset Y$, or more generally, $L$ equipped with a unitary rank 1 local system $\xi$. More precisely, assuming that $\pi:Y\rightarrow B$ does not admit any singular fiber, given $(L,\xi)$, the construction of $\cite{kf}$ and $\cite{ma}$ provides a family Floer module $\mathcal{F}(L,\xi)$, whose cohomology sheaf $H^\ast\mathcal{F}(L,\xi)$ can be realized as a (twisted) coherent sheaf on $Y^\vee$. When singular fibers appear in the SYZ fibration $\pi$, the construction of Abouzaid-Fukaya can be applied over every chamber $B_\alpha\subset B$ so that $\pi$ restricted to $\pi^{-1}(B_\alpha)$ is a Lagrangian torus bundle, this yields a coherent sheaf on the algebraic torus $(\mathbb{K}^\ast)^n$, which is regarded as a chart $U_\alpha^\vee$ of the mirror. If we know how to glue the charts $\left\{U_\alpha^\vee\right\}_{\alpha\in A}$ together to obtain the corrected mirror manifold $Y^\vee$, this approach will provide a way to verify the proposal of $\cite{ly}$ rigorously.\bigskip

However, the classical SYZ approach to mirror construction is not easy to carry out. The difficulties arise both in the construction of Lagrangian fibrations and in determining very complicated quantum corrections. Fortunately, there are several interesting special cases where this approach does work. Constructions of this kind are based on the fundamental paper of Auroux $\cite{da1}$, and later generalized in $\cite{aak,fooo2,cll}$. In all these cases, whenever quantum corrections arise, they can be explicitly determined via algebraic counts of stable holomorphic discs defined by Fukaya-Oh-Ohta-Ono in $\cite{fooo1}$; and whenever walls appear in $B$, they are of the form $\Delta\times\mathbb{R}_{>-\varepsilon}$, where $\Delta$ is a codimension 1 set in $B$ over which the singular Lagrangian fibers lie, see Appendix \ref{section:converse}. In particular, since walls disjoint from each other, more complicated scattering phenomenon does not appear. Our verifications in this paper depend heavily on these constructions.

\subsection{Statement of the results}
Let $\overline{X}$ be a smooth toric Calabi-Yau manifold with complex dimension $n$, which means that it's a toric variety with trivial canonical bundle $K_{\overline{X}}\cong\mathcal{O}_{\overline{X}}$. Consider the open Calabi-Yau manifold $X=\overline{X}\setminus D$, where $D$ is a smooth divisor in $\overline{X}$ to be specified in Section $\ref{section:gghl}$. The mirror construction in this case is essentially carried out by Gross in $\cite{mg2}$. It follows that the SYZ mirror of $X$ is the open Calabi-Yau manifold defined by
\begin{equation}\label{eq:conic}X^\vee=\left\{(x_1,\cdot\cdot\cdot,x_{n-1},y,z)\in(\mathbb{K}^\ast)^{n-1}\times\mathbb{K}^2|yz=g(x_1,\cdot\cdot\cdot,x_{n-1})\right\},\end{equation}
for some regular function $g$ on $(\mathbb{K}^\ast)^{n-1}$. Unless otherwise specified, we assume $\mathbb{K}=\mathit{\Lambda}$ (or $\mathbb{C}$ whenever there is no convergence issue) where
\begin{equation}\mathit{\Lambda}=\left\{\sum_{i=0}^\infty a_iT^{\lambda_i}|a_i\in\mathbb{C},\lambda_i\in\mathbb{R},\lambda_i\rightarrow\infty\right\}\end{equation}
is the \textit{Novikov field}. By projecting to the first $n-1$ coordinates, we get a fibration on $X^\vee$ by affine conics in $\mathbb{K}^2$:
\begin{equation}\label{eq:aff-conic}p_0:X^\vee\rightarrow(\mathbb{K}^\ast)^{n-1}.\end{equation}
As suggested in $\cite{ly}$, one should expect that $X$ admits a twin Lagrangian fibration. In Section $\ref{section:twin-tcy}$, we will show that this is indeed the case: the Lagrangian fibration $\pi_G:X\rightarrow B$ introduced by Goldstein and Gross $\cite{eg,mg2}$ and a non-proper Lagrangian fibration $\pi_H$ which we will define in Section $\ref{section:gghl}$ form a twin Lagrangian fibration on $X$. Moreover, this twin Lagrangian fibration structure
\begin{equation}B\xleftarrow{\pi_G}X\xrightarrow{\pi_H}B_\star\end{equation}
on $X$ is shown to be naturally induced from the affine conic bundle structure $(\ref{eq:aff-conic})$ on its mirror $X^\vee$ in the following sense.
\begin{definition}\label{definition:twin-ind}
Let $Y$ be a Calabi-Yau manifold and $Y^\vee$ be its mirror. We say that a Lagrangian fibration $\pi_\star:Y\rightarrow B_\star$ is mirror to the fibration $\rho:Y^\vee\rightarrow S$ by rigid analytic subvarieties on $Y^\vee$ if for every regular fiber $L_\star$ of $\pi_\star$, equipped with any $\xi_\star\in H^1(L_\star,U_\mathbb{K})$,
\begin{equation}\mathrm{supp}H^\ast\mathcal{F}(L_\star,\xi_\star)=\rho^{-1}(s),\end{equation}
where $\rho^{-1}(s)\subset Y^\vee$ is a fiber of $\rho$. If $\pi_\star$ and the SYZ fibration $\pi:Y\rightarrow B$ form a twin Lagrangian fibration on $Y$, then we say that this twin Lagrangian fibration is induced by the fibration $\rho:Y^\vee\rightarrow S$.
\end{definition}

\paragraph{Remark} It's easy to see that $(L_\star,\xi_\star)\mapsto\mathrm{supp}H^\ast\mathcal{F}(L_\star,\xi_\star)$ is a slight generalization of the Fourier type transformations introduced in $\cite{ap}$ and $\cite{lyz}$. The coherent sheaf $H^\ast\mathcal{F}(L_\star,\xi_\star)$ is in general not globally defined, and we need some gluing and analytic continuation procedure to make $\mathrm{supp}H^\ast\mathcal{F}(L_\star,\xi_\star)$ into a well-defined rigid analytic subvariety of $Y^\vee$, this will be discussed in detail in Section \ref{section:trans}.\\
In the above definition, whether the Lagrangian fibration $\pi_\star$ is proper depends on whether the fibration $\rho$ on the mirror is proper. Although for most of the examples considered in this paper, $\pi_\star$ is a non-proper Lagrangian fibration, by taking $Y$ to be the quotient of our examples by a certain real lattice in $\mathbb{R}^{n-1}$, it's not difficult to get an induced Lagrangian torus fibration on $Y$ which is mirror to a fibration on $Y^\vee$ by abelian subvarieties. The mirror symmetry of these manifolds is studied in Section 10.2 of $\cite{aak}$.\\
Since the heuristic argument which predicts the existence of a twin Lagrangian fibration on certain symplectic manifolds depends on a good understanding of certain moduli spaces of sheaves supported on the fibers of $\rho$, there is usually some mild assumptions on the singular fibers of $\rho$, see Section \ref{section:twin-ell}.\bigskip

Here we use the notation $U_\mathbb{K}$ to denote the unitary group $U(1)$ when $\mathbb{K}=\mathbb{C}$ or
\begin{equation}U_{\mathit{\Lambda}}=\left\{a_0+\sum_{i=1}^\infty a_iT^{\lambda_i}|a_0\neq0,\lambda_i>0\right\}\end{equation}
when $\mathbb{K}=\mathit{\Lambda}$.\\
With the above notation, our first result can be stated as follows.
\begin{theorem}\label{theorem:1.1}
Let $\overline{X}$ be an $n$-dimensional toric Calabi-Yau manifold, then there is a twin Lagrangian fibration $B\xleftarrow{\pi_G}X\xrightarrow{\pi_H}B_\star$ on $X=\overline{X}\setminus D$ of index 1. Moreover, this twin Lagrangian fibration is induced from the affine conic bundle structure $p_0:X^\vee\rightarrow(\mathbb{K}^\ast)^{n-1}$ in the sense of Definition \ref{definition:twin-ind}.
\end{theorem}
In the above, the \textit{index} of a twin Lagrangian fibration is defined to be the minimal codimension of the intersection locus (as a smooth submanifold in an SYZ fiber) between the fibers of $\pi_G$ and $\pi_H$. 

Similar constructions can be done in the converse direction, by endowing the mirror $\overline{X}^\vee$ with a suitable symplectic structure $\omega_\varepsilon$ to be specified below. Algebraically, $\overline{X}^\vee$ is a partial compactification of $X^\vee$ by a divisor $D^\vee$, that can be realized as the blow up of $V\times\mathbb{C}$ along the codimension 2 subvariety $H\times0\subset V\times\mathbb{C}$, where $V$ is an $(n-1)$-dimensional toric variety and $H\subset V$ is a smooth hypersurface with defining equation
\begin{equation}g(\mathbf{x})=0,\mathbf{x}=(x_1,\cdot\cdot\cdot,x_{n-1})\in V.\end{equation}
In this case, a piecewise smooth Lagrangian fibration $\pi_A:X^\vee\rightarrow B^\vee$ is constructed by Abouzaid-Auroux-Katzarkov in $\cite{aak}$. With some additional assumptions on $V$ and $H$ to exclude higher order instanton corrections, the SYZ mirror construction can be carried out in a similar way as outlined in $\cite{da1}$, so we obtain a rigid analytic Calabi-Yau manifold $X$, such that $\overline{X}=X\sqcup D$ is a toric Calabi-Yau variety. Here, $X$, $\overline{X}$ and $D$ are the same as before. In this case, the defining section $w$ of the toric boundary divisor $D_{\overline{X}}$ defines a fibration
\begin{equation}\label{eq:cpxfib1}w_0:X\rightarrow\mathbb{K}^\ast\end{equation}
whose generic fibers are hypersurfaces in $X$ isomorphic to $(\mathbb{K}^\ast)^{n-1}$, and $w_0$ can be identified with $w$ by adding some constant. According to $\cite{ly}$, such an additional geometric structure on $X$ should induce a twin Lagrangian fibration on $X^\vee$. In Section $\ref{section:lag-blp}$, we will construct a non-proper piecewise smooth Lagrangian fibration
\begin{equation}\label{eq:algtr-fib}\pi_L:X^\vee\rightarrow B_\star^\vee,\end{equation}
and show that $\pi_A$ and $\pi_L$ form the expected twin Lagrangian fibration on $X^\vee$ in Section $\ref{section:twin-blp}$.
\begin{theorem}\label{theorem:1.2}
Let $\overline{X}^\vee$ be the blow up of the toric variety $V\times\mathbb{C}$ along $H\times0$, where $H\subset V$ is a nearly tropical hypersurface in the sense of Definition \ref{definition:neartrop} and $V$ satisfies Assumption \ref{assumption:3.1}. Then there is a twin Lagrangian fibration $B^\vee\xleftarrow{\pi_A}X^\vee\xrightarrow{\pi_L}B_\star^\vee$ on $X^\vee$ of index $n-1$. Moreover, this twin Lagrangian fibration is induced from the fibration $w_0:X\rightarrow\mathbb{K}^\ast$ on the mirror manifold $X$.
\end{theorem}

The results above can be generalized to the intermediate cases when $H\subset V$ is a complete intersection, then the geometric structure on $X$ can be $T^k$-invariant for any $1\leq k\leq n-1$. These results should be compared with a speculation of Teleman $\cite{ct}$ on the mirror of the abelian gauge theory, see also $\cite{ps1}$. More explicitly, a Hamiltonian $T$-action on a symplectic manifold $Y$ is believed to be mirror to a holomorphic map $\rho:Y^\vee\rightarrow T_\mathbb{C}^\vee$, where $T_\mathbb{C}^\vee$ is the dual complexified torus of $T$. In our case, $\rho$ is not only a regular map but actually a fibration, which should then lead to more delicate geometric structures on $Y$, namely a twin Lagrangian fibration which is compatible with the Hamiltonian $T$-action.\bigskip

When $\dim_\mathbb{C}(X^\vee)=2$, the mirror symmetry between the fibrations $\pi_L$ and $w_0$ established in Theorem \ref{theorem:1.2} gives geometric understandings of homological mirror symmetry for $\overline{X}^\vee$. In Section \ref{section:fuk}, we introduce a Fukaya category $\mathcal{F}uk(\pi_L)$ associated to the Lagrangian fibration $\pi_L$ on $X^\vee$ which captures the information coming from the singular fibers of $\pi_L$. There is then an equivalence
\begin{equation}\label{eq:hmsint}D\mathcal{F}uk(\pi_L)\cong D^\pi_\mathrm{sing}\big(w^{-1}(0)\big),\end{equation}
where $D^\pi_\mathrm{sing}\big(w^{-1}(0)\big)$ is the idempotent completion of the triangulated category of singularities associated to the fibration $w:X\rightarrow\mathbb{K}$, defined by Orlov in $\cite{do}$. This very simple equivalence provides further evidence to the fact that the fibrations $\pi_L$ and $w_0$ are mirror to each other.

Motivated by Section 4 of Smith's very beautiful paper $\cite{is1}$, we consider the partial compactification $\mathit{Bl}_K(\mathbb{C}^2)$ of $X^\vee$, which is just $\mathbb{C}^2$ blown up at a finite set of points $K\subset\mathbb{C}^2$. We show that in $\mathit{Bl}_K(\mathbb{C}^2)$ (equipped with a suitable monotone symplectic form) there are a finite number of Lagrangian tori $\widetilde{T}_1,\cdot\cdot\cdot,\widetilde{T}_p$ which split-generate the non-zero eigensummand of the monotone Fukaya category $\mathcal{F}uk\big(\mathit{Bl}_K(\mathbb{C}^2)\big)$ (see Proposition \ref{proposition:generation}). In this case, the thimbles $\mathit{\Delta}_1,\cdot\cdot\cdot,\mathit{\Delta}_p$ retains under the fiberwise partial compactification
\begin{equation}\widetilde{p}_0:\mathit{Bl}_K(\mathbb{C}^2)\rightarrow\mathbb{C}\end{equation}
of the Lefschetz fibration $p_0$ on the smoothing of $A_{p-1}$ singularities $X_{p-1}^\vee\supset X^\vee$, and they are isomorphic to the idempotents of $\widetilde{T}_1,\cdot\cdot\cdot,\widetilde{T}_p$ in $\mathit{Tw}\mathcal{F}uk(\pi_L)$ up to degree shifts. These considerations lead to the following result, which shows that the equivalence $(\ref{eq:hmsint})$ fits naturally into a commutative diagram, and in particular interprets the homological mirror symmetry for $\mathit{Bl}_K(\mathbb{C}^2)$ as the mirror symmetry between the fibration structures $\pi_L$ and $w_0$.
\begin{theorem}\label{theorem:fuk}
Let $D^\pi(X,W)$ be the split-closed triangulated category of D-branes of type B defined by Orlov $\cite{do}$. The diagram
\begin{equation}
\xymatrix{
D\mathcal{F}uk(\pi_L) \ar[d]_{\Phi_\bullet} \ar[r]^-{\Phi(e^+)}
 &D^\pi\mathcal{F}uk\big(\mathit{Bl}_K(\mathbb{C}^2)\big)\ar[d]^{\Phi_\mathrm{CHL}}\\
D^\pi_\mathrm{sing}\big(w^{-1}(0)\big) \ar[r]^-{\Phi(e^+)^\vee}
 &D^\pi(X,W)}
\end{equation}
commutes, where $\Phi_\bullet$ is the equivalence $(\ref{eq:hmsint})$, and $\Phi_\mathrm{CHL}$ is some variation of the localized mirror functor defined in $\cite{chl,chl1}$.
\end{theorem}
The definitions of the functors $\Phi(e^+)$ and $\Phi(e^+)^\vee$ will be given in Section \ref{section:fuk}.\bigskip

In the context of family Floer theory, an obvious application of a twin Lagrangian fibration is that one can use the additional Lagrangian fibration to detect the non-displaceability of certain SYZ fibers. However, when singular fibers are involved, this is technically not easy. In this paper, we look at a particularly simple example, based on the observation that the singular fibers of $\pi_L$ contain a basis of Lefschetz thimbles $\mathit{\Delta}_1,\cdot\cdot\cdot,\mathit{\Delta}_p$ of $p_0$. By studying the Floer cohomologies between $\mathit{\Delta}_i$ and the Lagrangian torus fibers of $\pi_A$, we are able to detect the non-displaceable Lagrangian tori in
\begin{equation}X_{p-1}^\vee=\left\{(x,y,z)\in\mathbb{C}^3|yz=x^p-1\right\}.\end{equation}
Combining this with the work of Wu $\cite{www}$ on finite group actions on Fukaya categories, we can then compute the Floer cohomologies of certain Lagrangian tori $T_{p,q}$ in the rational homology balls $B_{p,q}=X_{p-1}^\vee/G_{p,q}$ with $G_{p,q}\cong\mathbb{Z}_p$, see Section \ref{section:ga}. This recovers the following result due to Lekili-Maydanskiy.
\begin{theorem}[Lekili-Maydanskiy $\cite{lm}$]\label{theorem:ga}
There exist Floer theoretical essential tori $T_{p,q}\subset B_{p,q}$, and
\begin{equation}\mathit{HF}^\ast(T_{p,q},T_{p,q})\cong H^\ast(T^2,\mathbb{K})\end{equation}
additively. In particular, the symplectic cohomology $\mathit{SH}^\ast(B_{p,q})\neq0$.
\end{theorem}
The Floer cohomologies $\mathit{HF}^\ast(T_{p,q},T_{p,q})$ are computed by Lekili and Maydanskiy earlier in $\cite{lm}$ over $\mathbb{Z}_2$ using the deep results of Biran and Cornea $\cite{bc}$ on pearl complexes.

\subsection{Arrangement of this paper}
The paper is organized as follows. Section $\ref{section:lagop}$ contains some preliminaries on constructing Lagrangian fibrations using symplectic reduction, and based on this we construct the Lagrangian fibrations $\pi_H$ and $\pi_L$. Using family Floer theory, we define the mirror transformation $(L,\xi)\mapsto\mathrm{supp}H^\ast\mathcal{F}(L,\xi)$ in Section \ref{section:trans}. This is the main tool we use in Section \ref{section:twin}, where we show that the twin Lagrangian fibrations on $X$ and $X^\vee$ are mirror to fibrations by rigid analytic subvarieties on $X^\vee$ and $X$ respectively. Section \ref{section:sing} focuses on the special case of complex surfaces, Theorems $\ref{theorem:fuk}$ and $\ref{theorem:ga}$ will be proved there. The background materials concerning SYZ mirror constructions will be used frequently in this paper, so we collect them in Appendix \ref{section:syz} for readers' convenience. The localized mirror functor $\widehat{\Phi}_\mathrm{CHL}$ introduced in $\cite{chl}$ is only used in proving Theorem \ref{theorem:fuk}, which will be briefly recalled in Appendix \ref{section:localized}.

\section*{Acknowledgements}
The second author would like to express his deep gratitude to Hansol Hong for calling his attention to their papers $\cite{chl}$ and $\cite{chl1}$ on localized mirror functors, and Weiwei Wu for very useful discussions on his paper $\cite{www}$ during his visit at CUHK. He also thanks Junwu Tu for telling him about the importance of rigid analytic geometry. We thank Hiroshi Ohta and Kaoru Ono for their interests on this work. The work of the first author is supported by grants from the Research Grants Council of the Hong Kong Special Administrative Region, China (Project No. CUHK402012 and CUHK14302215) and a Direct Grant from CUHK.

\section{Construction of Lagrangian fibrations}\label{section:lagop}

\subsection{Lagrangian fibrations on toric Calabi-Yau manifolds}\label{section:gghl}
Let $N\cong\mathbb{Z}^n$ be a lattice and $N^\vee$ be its dual lattice. $\Sigma$ is a strongly convex simplicial fan supported in $N_\mathbb{R}=N\otimes\mathbb{R}$. Associated to $\Sigma$ there is a smooth toric variety $\overline{X}:=\overline{X}_\Sigma$. Denote by $v_\alpha$ the primitive generators of the rays of $\Sigma$, then the Calabi-Yau condition $K_{\overline{X}}\cong\mathcal{O}_{\overline{X}}$ is equivalent to the existence of a $\nu\in N^\vee$ such that
\begin{equation}\langle\nu,v_\alpha\rangle=1\end{equation}
for every $v_\alpha$. We will denote the set which parametrizes $v_\alpha$ by $A$.
\bigskip

From now on we assume that $\overline{X}$ is a toric Calabi-Yau manifold with $\dim_\mathbb{C}(\overline{X})=n$. It's an easy observation that the meromorphic function $w:\overline{X}\rightarrow\mathbb{C}$ corresponding to $\nu\in N^\vee$ is actually holomorphic, therefore defines a global coordinate function. Let $T_\mathbb{C}(N)\subset \overline{X}$ be the maximal cell inside the toric variety $\overline{X}$, which is defined by $N\otimes_\mathbb{Z}\mathbb{C}^\ast$.

Embedded in $T_\mathbb{C}(N)$, there is a real torus $T_\mathbb{R}(N)$ which acts on $\overline{X}$ effectively, making $\overline{X}$ into a toric symplectic manifold. Define
\begin{equation}N_\nu=\left\{n\in N|\langle\nu,n\rangle=0\right\},\end{equation}
which determines an $(n-1)$-dimensional real torus $T_\mathbb{R}(N_\nu)$. Associated to the Hamiltonian $T_\mathbb{R}(N_\nu)$-action, there is a moment map
\begin{equation}\mu_\nu:\overline{X}\rightarrow\mathfrak{t}^\ast_\mathbb{R}(N_\nu)\cong\mathbb{R}^{n-1},\end{equation}
where $\mathfrak{t}_\mathbb{R}(N_\nu)$ is the Lie algebra of $T_\mathbb{R}(N_\nu)$. Denote by $D$ the anticanonical divisor
\begin{equation}\label{eq:antidiv1}\left\{w=-1\right\}\subset\overline{X},\end{equation}
and let $X=\overline{X}\setminus D$. A standard symplectic reduction argument reduces the problem of producing a special Lagrangian fibration on $X$ to that of producing a special Lagrangian fibration on each reduced space. The observation that these reduced spaces can be identified with the $w$ coordinate plane leads to the following theorem.
\begin{theorem}[Gross $\cite{mg2}$, Goldstein $\cite{eg}$]
Let $(\overline{X},\omega_{\overline{X}})$ be a toric Calabi-Yau manifold equipped with its toric K\"{a}hler form, then $\pi_G:X\rightarrow\mathbb{R}^n$ defined by
\begin{equation}\pi_G=\big(\log|w+1|,\mu_\nu\big)\end{equation}
is a special Lagrangian torus fibration with respect to the holomorphic volume form
\begin{equation}\Omega=\frac{\Omega_{\overline{X}}}{i^n(w+1)},\end{equation}
where $\Omega_{\overline{X}}=dz_1\wedge\cdot\cdot\cdot\wedge dz_n$ when restricted to $T_\mathbb{C}(N)$.
\end{theorem}
It's not difficult to verify that the set of critical points of $\pi_G$ is given by the union of codimension 2 toric strata in $\overline{X}$. From this we see that there is a codimension 2 discriminant locus $\Delta\subset B\simeq\mathbb{R}^n$ inside $\{0\}\times\mathbb{R}^{n-1}$ over which the fibers of $\pi_G$ become singular, see $\cite{mg2}$ for an explicit description. In particular, $\Delta$ can be decomposed as a disjoint union of smooth submanifolds. For example, when $\dim_\mathbb{C}(X)=3$, $\Delta\subset\mathbb{R}^3$ is a trivalent graph, and we have the decomposition
\begin{equation}\Delta=\Delta_d\sqcup\Delta_g\end{equation}
into vertices $\Delta_d$ and edges $\Delta_g$. The singular fiber $\pi_G^{-1}(\bullet)$ over $\bullet\in\Delta_g$ is obtained by collapsing a circle in the regular fiber $L$ down to a point. These are called \textit{generic} singular fibers. Passing from an edge to a vertex in $\Delta_d$ makes the corresponding Lagrangian fiber ''more singular", which means there is a 2-torus in $L$ which collapses to a point. Lagrangian fibrations with such a topological behavior near a trivalent vertex in $\Delta_d$ are called \textit{positive} in the sense of $\cite{mg1}$ and $\cite{cm2}$. The situation in higher dimensions is similar.\bigskip

Using the same method we get the following.
\begin{example}\label{theorem:ghlf}
The map $\pi_H:X\rightarrow S^1\times\mathbb{R}^{n-1}$ defined by
\begin{equation}\pi_H=\big(\arg(w+1),\mu_\nu\big)\end{equation}
is a Lagrangian fibration on $X$ with respect to $\omega_{\overline{X}}$.
\end{example}

It's clear from the definition that the Lagrangian fibration $\pi_H$ is smooth and non-proper, and its regular fibers are homeomorphic to $T^{n-1}\times\mathbb{R}$. On the other hand, it's easy to see that with the above definition, the set of critical points of the map $\pi_H$ coincides with that of $\pi_G$, from which we obtain an identification between the discriminant loci of the Lagrangian fibrations $\pi_H$ and $\pi_G$. 

The singular fibers of $\pi_H$ has a similar description with that of $\pi_G$. For a generic singular fiber $\pi_H^{-1}(\bullet)$ of $\pi_H$, it can be decomposed as
\begin{equation}\label{eq:singdecomp}\pi_H^{-1}(\bullet)=L^+\cup L^-,\end{equation}
where $L^\pm$ are Lagrangian submanifolds homeomorphic to $T^{n-2}\times\mathbb{R}^2$ and $L^+$ intersects $L^-$ cleanly along a $T^{n-2}$. Over lower dimensional components of $\Delta$ (if non-empty), the fibers of $\pi_H$ become ``more singular" in the sense that one of the orbits of the Hamiltonian $T_\mathbb{R}(N_\nu)$-action degenerates to a lower-dimensional torus $T^k$ with $k<n-2$. In particular, over the vertices in $\Delta$, the fiber becomes a cone over $T^{n-1}$.\bigskip

The following follows easily from the definitions of the Lagrangian fibrations $\pi_G$ and $\pi_H$.

\begin{proposition}\label{proposition:inter}
Let $L$ be a regular fiber of $\pi_G$, then for any fiber $L_\star$ of $\pi_H$ which intersects $L$ non-trivially, the intersection $L\cap L_\star$ is an $(n-1)$-dimensional sub-torus of $L$. Similarly, for any fiber $L$ of $\pi_G$ which has non-trivial intersection with a regular fiber $L_\star$ of $\pi_H$, it intersects $L_\star$ cleanly along a $T^{n-1}$.
\end{proposition}
\begin{proof}
	Fix an $L$ so that
	\begin{equation}L=\big\{\left|w+1\right|=C_1,\mu_\nu=C_2\big\},\end{equation}
	where $C_1>0$ and $C_2\in\mathbb{R}^{n-1}$ are constants. For any point in $L\cap L_\star$, it satisfies $\mu_\nu=C_2$, which implies that any $L_\star$ with $L\cap L_\star\neq\emptyset$ has the form
	\begin{equation}L_\star=\big\{\arg\left(w+1\right)=C_3,\mu_\nu=C_2\big\}\end{equation}
	for some $C_3\in S^1$. For the reduced coordinate $w$, the ray specified by $\arg(w+1)=C_3$ intersects the circle $|w+1|=C_1$ precisely at one point. This shows that $L\cap L_\star\simeq T^{n-1}$, where the intersection locus is a Hamiltonian orbit of the $T_\mathbb{R}(N_\nu)$-action.
	
	Similar argument applies to any regular fiber $L_\star$ of $\pi_H$.
\end{proof}

\subsection{Lagrangian fibrations on affine conic bundles}\label{section:lag-blp}
Let $V$ be an $(n-1)$-dimensional toric variety, $H\subset V$ a \textit{nearly tropical} hypersurface (i.e. $H$ belongs to a certain degenerating family, $H_\tau$, converging to certain tropical limit, which we will define below). Denote by $\overline{X}^\vee$ the blow up of $V\times\mathbb{C}$ along $H\times0$. Here we follow closely the approach of $\cite{aak}$ to introduce the A-side geometric setup of $\overline{X}^\vee$.\\
Let $\Sigma_V\subset\mathbb{R}^{n-1}$ be the fan associated to $V$, which is generated by a set of primitive integral vectors $\sigma_1,\cdot\cdot\cdot,\sigma_r$. Let $H_\tau\subset V$ be a family of smooth algebraic hypersurfaces with $0<\tau<1$, and assume that they are transverse to the toric boundary divisor $D_V\subset V$. To understand the combinatorial nature of the hypersurfaces $H_\tau\subset V$, we can pass to the \textit{tropical limit} $\tau\rightarrow0$ and look at the degeneration of $H_\tau$. More precisely, let $\mathbf{x}=(x_1,\cdot\cdot\cdot,x_{n-1})$ be the coordinates of the open stratum $(\mathbb{C}^\ast)^{n-1}\subset V$, which we denote by $V_0$. Suppose that $H_\tau$ are defined by the following equations
\begin{equation}\label{eq:trohp}g_\tau(\mathbf{x})=\sum_{\alpha\in A}c_\alpha\tau^{\rho(\alpha)}\mathbf{x}^\alpha=0,\end{equation}
where $A\subset\mathbb{Z}^{n-1}$ consists of the group characters of $V_0$, $c_\alpha\in\mathbb{C}^\ast$ and $\rho:A\rightarrow\mathbb{R}$ is a map satisfying certain convexity property, which we will define shortly below.\\
An alternative way to describe $H_\tau$ is to regard it as the zero locus of a section $g_\tau$ of some nef line bundle $\mathcal{L}\rightarrow V$ defined by the convex piecewise linear function $\ell:\Sigma_V\rightarrow\mathbb{R}$ with integer slopes. The polytope $P_\mathcal{L}$ associated to $\mathcal{L}$ is given by
\begin{equation}P_\mathcal{L}=\left\{v\in\mathbb{R}^{n-1}|\langle\sigma_i,v\rangle+\ell(\sigma_i)\geq0,\forall1\leq i\leq r\right\},\end{equation}
then $P_\mathcal{L}\cap\mathbb{Z}^{n-1}$ can be identified with a basis of $H^0\big(X,\mathcal{O}(\mathcal{L})\big)$. The condition that $H_\tau$ is transversal to $D_V$ is equivalent to the requirement that $A\subset P_\mathcal{L}\cap\mathbb{Z}^{n-1}$ intersects nontrivially with the closure of each face of $P_\mathcal{L}$.\\
We want to impose an additional convexity assumption on $\rho$. To do this, let $\mathcal{P}$ be a polyhedral decomposition of the convex hull $\mathrm{Conv}(A)\subset\mathbb{R}^{n-1}$, whose set of vertices is given by $\mathcal{P}^{(0)}=A$. We assume further that $\mathcal{P}$ is \textit{regular}, i.e. every cell of $\mathcal{P}$ is congruent to a standard simplex under the action of $\mathit{GL}(n-1,\mathbb{Z})$. The map $\rho:A\rightarrow\mathbb{R}$ is said to be \textit{adapted} to $\mathcal{P}$ if it's the restriction of a convex piecewise linear function $\tilde{\rho}:\mathrm{Conv}(A)\rightarrow\mathbb{R}$ on $A$, whose maximal domains of linearity are exactly those cells of $\mathcal{P}$.
\begin{definition}[$\cite{aak}$]
We say that the family of hypersurfaces $H_\tau\subset V$ has a maximal degeneration for $\tau\rightarrow0$ if it is defined by $(\ref{eq:trohp})$ and $\rho$ is adapted to some regular polyhedral decomposition $\mathcal{P}$.
\end{definition}
For every fixed hypersurface $H_\tau$, consider the image of $H_\tau\cap V_0$ under the map
\begin{equation}\label{eq:logmap}\mathrm{Log}_\tau:(x_1,\cdot\cdot\cdot,x_{n-1})\mapsto\frac{1}{|\log\tau|}\big(\log|x_1|,\cdot\cdot\cdot,\log|x_{n-1}|\big).\end{equation}
This is known as an \textit{amoeba} $\mathit{\Pi}_\tau\subset\mathbb{R}^{n-1}$. As $\tau\rightarrow0$, $\mathit{\Pi}_\tau\subset\mathbb{R}^{n-1}$ converges to a tropical hypersurface $\mathit{\Pi}_0\subset\mathbb{R}^{n-1}$ defined by the tropical polynomial
\begin{equation}\label{eq:tropp}\chi(\xi)=\max\big\{\langle\alpha,\xi\rangle-\rho(\alpha)|\alpha\in A\big\}.\end{equation}
In fact, $\mathit{\Pi}_0$ is just the dual cell complex of $\mathcal{P}$, in particular the connected components of $\mathbb{R}^{n-1}\setminus\mathit{\Pi}_0$ are labeled by the elements of $\mathcal{P}^{(0)}=A$, depending on which term in $(\ref{eq:tropp})$ is the maximal one.
\begin{definition}[$\cite{aak}$]\label{definition:neartrop}
A smooth hypersurface $H\subset V$ is called nearly tropical if it appears as a member of a maximally degenerating family of hypersurfaces, with the additional property that its amoeba $\mathit{\Pi}=\mathrm{Log}(H\cap V_0)$ is contained in a neighborhood of the tropical hypersurface $\mathit{\Pi}_0$, and there is a retraction from $\mathit{\Pi}$ to $\mathit{\Pi}_0$.
\end{definition}
Roughly speaking, nearly tropical means that $H$ is close enough to its tropical limit, so that the complement of $\mathit{\Pi}_\tau$ in $\mathbb{R}^{n-1}$ have same combinatorial type as that of $\mathbb{R}^{n-1}\setminus\mathit{\Pi}_0$ in the sense that $\mathbb{R}^{n-1}\setminus\mathit{\Pi}_\tau$ and $\mathbb{R}^{n-1}\setminus\mathit{\Pi}_0$ have the same number of chambers and the adjacency between chambers is preserved when passing to the limit $\tau\rightarrow0$. The assumption that the family is maximally degenerating is intended to ensure that the mirror $X$ of $X^\vee$ constructed in Appendix \ref{section:converse} is smooth.\\
To equip $\overline{X}^\vee$ with an appropriate symplectic structure, we first write down the equation of $\overline{X}^\vee$ using the coordinates on $V\times\mathbb{C}$ and the fiber coordinates of $\mathcal{L}$. Recall that the defining equation $g(\mathbf{x})$ of $H$ can be identified with a section of the line bundle $\mathcal{L}\rightarrow V$. The normal bundle $\nu_{H\times0}\subset V\times\mathbb{C}$ is given by $(\mathcal{L}\times\mathbb{C})\oplus\mathcal{O}_{V\times\mathbb{C}}|_{H\times0}$, so we can realize $\overline{X}^\vee$ as a hypersurface in the total space of the fiberwise compactification
\begin{equation}\mathbb{P}\big((\mathcal{L}\times\mathbb{C})\oplus\mathcal{O}_{V\times\mathbb{C}}\big)\rightarrow V\times\mathbb{C}.\end{equation}
More explicitly, denote by $\mathbf{x}$ and $y$ the coordinates on $V$ and $\mathbb{C}$ respectively, then
\begin{equation}\overline{X}^\vee=\Big\{\big(\mathbf{x},y,(u:v)\big)\in\mathbb{P}\big((\mathcal{L}\times\mathbb{C})\oplus\mathcal{O}_{V\times\mathbb{C}}\big)|g(\mathbf{x})v=yu\Big\}.\end{equation}
Note that this is a partial compactification of the affine conic bundle $X^\vee$ defined in $(\ref{eq:conic})$ by the anticanonical divisor
\begin{equation}D^\vee=p^{-1}(D_V\times\mathbb{C})\cup\widetilde{V},\end{equation}
where $p:\overline{X}^\vee\rightarrow V\times\mathbb{C}$ denotes the blow up map and $\widetilde{V}$ is the proper transform of $V$. Consider the following $S^1$-action on $\overline{X}^\vee$:
\begin{equation}\label{eq:ciract}e^{i\theta}\cdot\big(\mathbf{x},y,(u:v)\big)=\big(\mathbf{x},e^{i\theta}y,(u:e^{i\theta}v)\big).\end{equation}
This action preserves the exceptional divisor $E$ and acts by rotation on each fiber of the trivial $\mathbb{P}^1$-bundle
\begin{equation}\label{eq:p1bun}p|_E:E\rightarrow H\times0.\end{equation}
Also the the fixed point set of the $S^1$-action $(\ref{eq:ciract})$ is given by $\widetilde{V}\sqcup \widetilde{H}$, where $\widetilde{H}$ consists of the points $(0:1)$ in each fiber of $(\ref{eq:p1bun})$.\\
To equip $\overline{X}^\vee$ with an $S^1$-invariant K\"{a}hler form $\omega_\varepsilon$, $\cite{aak}$ introduces an $S^1$-invariant $C^\infty$ cutoff function $\eta:\overline{X}^\vee\rightarrow\mathbb{R}$ with $\mathrm{supp}(\eta)$ lying inside a tubular neighborhood of $H\times0$ and $\eta=1$ near $H\times0$. Set
\begin{equation}\label{eq:kahfor}\omega_\varepsilon=p^\ast\omega_{V\times\mathbb{C}}+\frac{i\varepsilon}{2\pi}\partial\bar{\partial}\Big(\eta(\mathbf{x},y)\log\big(|g(\mathbf{x})|^2+|y|^2\big)\Big).\end{equation} This is a well-defined $S^1$-invariant K\"{a}hler form on $\overline{X}^\vee$ provided that $\varepsilon>0$ is small enough. More precisely, $\varepsilon$ needs to be chosen so that a standard symplectic neighborhood of size $\varepsilon$ of $H\times0$ can be embedded $S^1$ equivariantly into $\mathrm{supp}(\eta)$. To achieve this, the following assumption is imposed in $\cite{aak}$:
\begin{assumption}\label{assumption:symnei}
$\mathrm{supp}(\eta)\subset p^{-1}(U_H\times B_\delta)$, where $U_H\supset H$ is a standard symplectic neighborhood of $H$ with area $\delta$ and $B_\delta\subset\mathbb{C}$ is the disc of radius $\delta$.
\end{assumption}
We now review the construction of the Lagrangian torus fibration $\pi_A$ on $(X^\vee,\omega_\varepsilon)$ in $\cite{aak}$. First note that the $S^1$-action $(\ref{eq:ciract})$ on $\overline{X}^\vee$ is Hamiltonian with respect to $\omega_\varepsilon$, denote by $\mu_1:\overline{X}^\vee\rightarrow\mathbb{R}$ the associated moment map. By the expression $(\ref{eq:kahfor})$ of $\omega_\varepsilon$, $\mu_1$ is given by
\begin{equation}\mu_1(\mathbf{x},y)=\pi|y|^2+\varepsilon|y|\frac{\partial}{\partial|y|}\Big(\eta(\mathbf{x},y)\log\big(|g(\mathbf{x})|^2+|y|^2\big)\Big).\end{equation}
Here we prefer to make a slight modification and set $\mu_0=\mu_1-\varepsilon$, so it takes the value $-\varepsilon$ over $\widetilde{V}$. The generic level sets of $\mu_0$ are smooth, with the exception that $\mu_0^{-1}(0)$ is singular along $\widetilde{H}\subset(\overline{X}^\vee)^{S^1}$.\\
Denote by
\begin{equation}\overline{X}^\vee_{\mathrm{red},\lambda}=\mu_0^{-1}(\lambda)/S^1\end{equation}
the reduced space at $\lambda$. For $\lambda>-\varepsilon$, we have a diffeomorphism $\overline{X}^\vee_{\mathrm{red},\lambda}\cong V$. Also for $\lambda\gg0$, since $\mu_0^{-1}(\lambda)$ is disjoint from $\mathrm{supp}(\eta)$, then by $(\ref{eq:kahfor})$, we have a symplectomorphism $(\overline{X}^\vee_{\mathrm{red},\lambda},\omega_{\mathrm{red},\lambda})\cong(V,\omega_V)$. But for $\lambda<0$, $\omega_{\mathrm{red},\lambda}$ differs from the toric K\"{a}hler form $\omega_V$ in a tubular neighborhood of $H$.\\
The remedy is to average $\omega_{\mathrm{red},\lambda}$ with respect to the standard $T^{n-1}$-action on the toric variety $V$:
\begin{equation}\omega_{V,\lambda}=\frac{1}{(2\pi)^n}\int_{T^{n-1}}t^\ast\omega_{\mathrm{red},\lambda}dt,\end{equation}
which leads to a toric K\"{a}hler form $\omega_{V,\lambda}$ for $\lambda\neq0$. Since $\omega_{\mathrm{red},\lambda}$ and $\omega_{V,\lambda}$ lie in the same cohomology class, the following Moser type lemma can be obtained.
\begin{lemma}[Abouzaid-Auroux-Katzarkov \cite{aak}, Lemma 4.1]\label{lemma:aak}
There exists a family of homeomorphisms $\phi_\lambda:(\overline{X}^\vee_{\mathrm{red},\lambda},\omega_{\mathrm{red},\lambda})\rightarrow(V,\omega_{V,\lambda})$ for $\lambda\in\mathbb{R}_{>-\varepsilon}$ such that $\phi_\lambda^\ast\omega_{V,\lambda}=\omega_{\mathrm{red},\lambda}$ and
\begin{itemize}
\item $\phi_\lambda$ preserves $D_V\subset V$;
\item for $\lambda\neq0$, $\phi_\lambda$ restricted to $V_0$ is a diffeomorphism;
\item $\phi_\lambda$ depends piecewise smoothly on $\lambda$, and smoothly except when $\lambda=0$.
\end{itemize}
\end{lemma}
This is the key lemma which enables us to complete the construction of a Lagrangian torus fibration on $X^\vee$. Namely one first applies the diffeomorphism $\phi_\lambda$ to identify $\overline{X}^\vee_{\mathrm{red},\lambda}$ with the toric symplectic manifold $(V,\omega_{V,\lambda})$, then the map $\mathrm{Log}_\tau$ defined by $(\ref{eq:logmap})$ will induce a Lagrangian torus fibration on $V_0$, which together with the Hamiltonian $S^1$ orbits gives a Lagrangian fibration on $X^\vee$.
\begin{theorem}[Abouzaid-Auroux-Katzarkov \cite{aak}]\label{theorem:aak}
The map $\pi_A:X^\vee\rightarrow\mathbb{R}^{n-1}\times\mathbb{R}_{>-\varepsilon}$ defined by
\begin{equation}\pi_A(\mathbf{p})=\big(\mathrm{Log}_\tau\circ\phi_\lambda(\mathbf{x}),\mu_0(\mathbf{p})=\lambda\big),\end{equation}
where $\mathbf{x}\in \overline{X}^\vee_{\mathrm{red},\lambda}$ is the $S^1$ orbit of $\mathbf{p}\in X^\vee$, is a Lagrangian torus fibration on $X^\vee$ with respect to the symplectic form $\omega_\varepsilon$.
\end{theorem}

In general, the fibration $\pi_A$ is only piecewise smooth when $\dim_\mathbb{C}(X^\vee)\geq3$. In fact, this coincides with various expectations in the study of mirror symmetry from the vewpoint of $T$-duality, see for example $\cite{mg1}$, $\cite{dj}$ and $\cite{cm2}$. More precisely, according to $\cite{mg1}$, $T$-duality in dimension 3 is topologically a duality between positive and \textit{negative} vertices of the discriminant locus. However, it seems near a negative vertex, the local model of and the expected Lagrangian torus fibration cannot be smooth, instead a piecewise smooth fibration can be constructed, see for example $\cite{cm2}$. In our case, every vertex in $\Delta_d$ is positive, and $\pi_G$ is a smooth fibration. This explains why, as the dual of $\pi_G$, the Lagrangian fibration $\pi_A$ should only be piecewise smooth.\bigskip

When $\dim_\mathbb{C}(X^\vee)\geq3$, the set of critical points of $\pi_A$ can be identified with the hypersurface $H\subset V_0$, therefore under the projection of the map $\mathrm{Log}_\tau\circ\phi_\varepsilon$, its discriminant locus $\Delta\subset B^\vee$ is an amoeba $\widetilde{\mathit{\Pi}}\subset\mathbb{R}^{n-1}\times\{0\}$ diffeomorphic to $\mathit{\Pi}_\tau$, and the generic singular fibers of $\pi_A$ over $\mathit{\Pi}_\tau$ are topologically circle fibrations over $T^{n-1}\subset V_0$, with the circle fibers over $H\subset V_0$ being collapsed to points.\bigskip

Using the same method we can construct a piecewise smooth non-proper Lagrangian fibration on $X^\vee$.
\begin{example}\label{example:piell}
The map $\pi_L:X^\vee\rightarrow T^{n-1}\times\mathbb{R}_{>-\varepsilon}$ defined by
\begin{equation}\pi_L(\mathbf{p})=\big(\mathrm{Arg}\circ\phi_\lambda(\mathbf{x}),\mu_0(\mathbf{p})=\lambda\big),\end{equation}
where
\begin{equation}\label{eq:image}\mathrm{Arg}(\mathbf{x})=\big(\arg(x_1),\cdot\cdot\cdot,\arg(x_{n-1})\big)\end{equation}
is a piecewise smooth Lagrangian fibration on $X^\vee$ with respect to $\omega_\varepsilon$, with its generic fiber $L_\star$ homeomorphic to $\mathbb{R}^{n-1}\times S^1$.
\end{example}
From now on, we shall fix the same choice of the family of homeomorphisms $\{\phi_\lambda\}$ in the construction of the Lagrangian fibrations $\pi_A$ and $\pi_L$.\bigskip

It's straightforward to verify that the set of critical points of $\pi_L$ coincides with that of $\pi_A$. From this we see that the discriminant locus of $\pi_L$ is the image of the hypersurface $H\subset V_0$ under the map $\mathrm{Arg}\circ\phi_\varepsilon$. A generic singular fiber of $\pi_L$ is topologically the singular space obtained by collapsing the $S^1$-orbits in $\mathbb{R}^{n-1}\times S^1$  of the Hamiltonian circle action over the hypersurface $H$ to points.\bigskip

\begin{proposition}\label{proposition:clean}
Let $L$ be a regular fiber of $\pi_A$, then for any fiber $L_\star$ of $\pi_L$ which intersects $L$ non-trivially, $L\cap L_\star\cong S^1$. Similarly, for any fiber $L$ of $\pi_A$ which has non-trivial intersection with a regular fiber $L_\star$ of $\pi_L$, the intersection is a circle.
\end{proposition}
\begin{proof}
	Consider a regular fiber $L$ of $\pi_A$, it be written as
	\begin{equation}L=\big\{\mathrm{Log}\circ\phi_{C_6}(\mathbf{x})=C_5,\mu_0(\mathbf{p})=C_6\big\},\end{equation}
	where $C_5\in\mathbb{R}^{n-1}$ and $C_6\in\mathbb{R}_{>-\varepsilon}$ are constants. If $L\cap L_\star\neq\emptyset$, the fiber $L_\star$ of $\pi_L$ must satisfy
	\begin{equation}L_\star=\big\{\mathrm{Arg}\circ\phi_{C_6}(\mathbf{x})=C_7,\mu_0(\mathbf{p})=C_6\big\},\end{equation}
	where $C_7\in T^{n-1}$. Since $\big\{\mathrm{Log}(\mathbf{x})=C_5\big\}$ and $\big\{\mathrm{Arg}(\mathbf{x})=C_7\big\}$ intersect transversally at one point in $V_0$, it's easy to see that $L\cap L_\star$ is a circle, which is in fact a Hamiltonian $S^1$-orbit of the action $(\ref{eq:ciract})$.
	
	Similar considerations hold for any regular fiber $L_\star$ of $\pi_L$.
\end{proof}

We now introduce the notion of a twin Lagrangian fibration, whose motivation from mirror symmetry will be discussed later in Section \ref{section:twin-ell}.
\begin{definition}\label{definition:4.1}
	A twin Lagrangian fibration on the symplectic manifold $Y$ consists of two Lagrangian fibrations $\pi:Y\rightarrow B$, $\pi_\star:Y\rightarrow B_\star$ (which may contain singular fibers), such that for every regular fiber $L$ of $\pi$, any fiber $L_\star$ (possibly singular) of $\pi_\star$ with $L\cap L_\star\neq\emptyset$ intersects $L$ cleanly along a smooth submanifold. The same is required for every regular fiber of $\pi_\star$. Such a structure will be denoted by
	\begin{equation}B\xleftarrow{\pi}Y\xrightarrow{\pi_\star}B_\star.\end{equation}
	The index of a twin Lagrangian is defined to be $\mathrm{codim}_\mathbb{R}(L\cap L_\star)$ in $L$ or $L_\star$.
\end{definition}

Our discussions in this section implies the following:
\begin{proposition}
The Lagrangian fibrations $\pi_G$ and $\pi_H$ form a twin Lagrangian fibration of index 1 on $X$, and the Lagrangian fibrations $\pi_A$ and $\pi_L$ form a twin Lagrangian fibration of index $(n-1)$ on $X^\vee$.
\end{proposition}

\section{Family Floer theory and mirror transformation}\label{section:trans}
In this section, we introduce a mirror transformation
\begin{equation}(L,\xi)\mapsto\mathrm{supp}H^\ast\mathcal{F}(L,\xi)\end{equation}
by applying the formalism of family Floer theory developed in $\cite{ma,ma1,kf}$. The expositions here follows $\cite{ma,ma1,aak,kf2}$.\\
We begin with the local case. Let $\pi:Y\rightarrow B$ be a Lagrangian torus bundle with a Lagrangian section, so that as integral affine manifolds
\begin{equation}(B,\mathcal{A})\cong(P,\mathcal{A}_\mathrm{std}),\end{equation}
where $P\subset\mathbb{R}^n$ is the interior of a convex polytope equipped with the standard integral affine structure. By Arnold-Liouville theorem, we have the identifications
\begin{equation}T_bB\cong H^1(F_b,\mathbb{R}),T_b^\ast B\cong H_1(F_b,\mathbb{R}),\end{equation}
where $F_b$ is any fiber of $\pi$. Assuming $\pi_2(B)=0$, the mirror $Y^\vee$ of $Y$ is the rigid analytic variety defined by
\begin{equation}\label{eq:defbb}Y^\vee\equiv\mathrm{val}^{-1}(B)=\bigsqcup_{b\in B}H^1(F_b,U_\mathbb{K})\subset H^1(F_b,\mathbb{K}^\ast),\end{equation}
where $\mathrm{val}:Y^\vee\rightarrow B$ denotes the valuation.
\begin{definition}
A Lagrangian submanifold $L$ in a symplectic maniold $Y$ is tautologically unobstructed if there exists a tame almost complex structure $J_L$ on $Y$ such that $L$ bounds no non-constant $J_L$-holomorphic disc.
\end{definition}
Let $L\subset Y$ be a tautologically unobstructed Lagrangian submanifold which is \textit{relatively Spin}, we recall the local construction of the \textit{family Floer module} $\mathcal{F}(L)$ associated to $L$ due to $\cite{ma,ma1}$.\\
For $F_b$ a fiber of $\pi:Y\rightarrow B$, assume that there is a (compactly supported) Hamiltonian diffeomorphism $\phi$ so that the intersection $F_b\cap\phi(L)$ is transverse for every $b\in B$ with $F_b\cap\phi(L)\neq\emptyset$. For $x,y\in F_b\cap\phi(L)$, denote by $\mathcal{M}_b(x,y)$ the moduli space of solutions $u:\mathbb{R}\times[0,1]\rightarrow Y$ of the equation
\begin{equation}(\partial_s-J_t\partial_t)u=0\end{equation}
with Lagrangian boundary conditions
\begin{equation}u(s,0)\in F_b,u(s,1)\in\phi(L)\end{equation}
and asymptotic conditions
\begin{equation}\lim_{s\rightarrow-\infty}u(s,t)=x,\lim_{s\rightarrow+\infty}u(s,t)=y,\end{equation}
where $\{J_t\}_{t\in[0,1]}$ is a family of tame almost complex structures so that $J_1=\phi_\ast(J_L)$. Associated to $u\in\mathcal{M}_b(x,y)$ there is an orientation line $o_u$ of the linearized Cauchy-Riemann operator $D_u$ at $u$.\\
Since $L$ is assumed to be relatively Spin, standard index theory assigns a rank 1 free abelian group $o_x$ to each intersection point $x\in F_b\cap\phi(L)$, and there is a canonical isomorphism
\begin{equation}o_u\otimes o_x\cong o_y.\end{equation}
Assuming $\deg y=\deg x+1$, for a generic choice of $\{J_t\}$ the moduli space $\mathcal{M}_b(x,y)$ consists only of rigid elements, and $\ker D_u$ is 1-dimensional. Fixing the orientation of $\ker D_u$ corresponding to the positive direction one gets a canonical map
\begin{equation}\partial_u:o_x\rightarrow o_y.\end{equation}
The local version of family Floer module is defined as
\begin{equation}\mathcal{F}(L)=\bigoplus_{x\in F_b\cap\phi(L)}\mathcal{O}_{Y^\vee}\otimes o_x.\end{equation}
To define the associated differential $\delta:\mathcal{F}(L)\rightarrow\mathcal{F}(L)[1]$, for each intersection point $x\in\phi(L)\cap F_b$, choose a function $g_x:B\rightarrow\mathbb{R}$ such that the Lagrangian section of $\pi:Y\rightarrow B$ is obtained by fiberwise addition of $dg_x$. This choice determines a path $\gamma$ on $F_b$ from $x$ to the basepoint, i.e. the intersection between $F_b$ and the fixed Lagrangian section. Define
\begin{equation}[\partial u]\in H_1(F_b,\mathbb{Z})\end{equation}
to be the homology class of the loop in $F_b$ obtained by concatenating the boundary of the strip $\partial u$ with $\gamma$. Define
\begin{equation}\delta|o_x:=\bigoplus_y\sum_{u\in\mathcal{M}_b(y,x)}T^{\mathcal{E}(u)}z^{[\partial u]}\otimes\partial_u,\end{equation}
where $\mathcal{E}(u)$ is the energy of $u$, and $z$ is the coordinate on $Y^\vee$. It's a non-trivial fact that the infinite sum $\sum_{u\in\mathcal{M}_b(y,x)}T^{\mathcal{E}(u)}z^{[\partial u]}$ converges in $T$-adic topology and defines a function in $\mathcal{O}_{Y^\vee}$, see $\cite{ma,kf2}$.\\
Passing from local to global requires carefully choosing the Floer data and establishing the relevant continuation maps. We are not going to recall these issues as they are not needed here, see $\cite{ma,ma1}$ for more details.\\
Passing to cohomology gives us a coherent sheaf $H^\ast\mathcal{F}(L)$ over $Y^\vee$, with its stalk over $(F_b,\xi_b)\in Y^\vee$ given by the Floer cohomology group $\mathit{HF}^\ast\big((F_b,\xi_b),L\big)$.\\
The above construction has an obvious generalization to the case when $L$ is equipped with a unitary rank 1 local system $\xi$, and the coherent sheaf $H^\ast\mathcal{F}(L,\xi)$ has its stalk over $(F_b,\xi_b)$ the Floer cohomology group
\begin{equation}\label{eq:flco}\mathit{HF}^\ast\big((F_b,\xi_b),(L,\xi)\big).\end{equation}
Notice that $L\subset Y$ is not assumed to be compact in the above definition, since $F_b$ is closed and both Lagrangian submanifolds are tautologically unobstructed, the Floer cohomology groups $(\ref{eq:flco})$ are well-defined in the usual sense.\bigskip

Now consider the more general case when $\pi:Y\rightarrow B$ is a Lagrangian fibration with possibly singular fibers but whose generic fiber has vanishing Maslov class. This is the case of $\pi_G$ and $\pi_A$ introduced above, see for example, Proposition 5.1 of $\cite{aak}$ for details. Let $\big\{(L_t,J_t)\big\}_{t\in[0,1]}$ be a path between the Lagrangian fibers $L_0,L_1\subset Y$, and $J_t$ is a family of almost complex structures which are fixed at infinity. For the cases we deal with in this paper, the following assumption is always satisfied for any two regular fibers $F_p$ and $F_q$ of $\pi$.
\begin{assumption}
The path $\big\{(L_t,J_t)\big\}_{t\in[0,1]}$ can be decomposed into finitely many sub-paths $\big\{(L_t,J_t)\big\}_{t\in[t_0,t_1]}$ so that all simple stable holomorphic discs lie in a fixed class in $H_2(Y,L_{t_0})$.
\end{assumption}
Under this assumption, we have a birational map
\begin{equation}H^1(L_0,\mathbb{K}^\ast)\dashrightarrow H^1(L_1,\mathbb{K}^\ast)\end{equation}
whose construction is essentially due to Fukaya in $\cite{kf2}$, see also $\cite{jt}$. After specializing to the case when $L_0=F_p$ and $L_1=F_q$, we get a wall-crossing map
\begin{equation}\label{eq:wall-crossing}\mathit{\Upsilon}_{\alpha\beta}:H^1(F_p,\mathbb{K}^\ast)\dashrightarrow H^1(F_q,\mathbb{K}^\ast).\end{equation}
In the cases treated in Appendix \ref{section:syz}, there exist charts $U_\alpha,U_\beta\subset Y$ fitting into the local picture above, such that $F_p\subset U_\alpha$, $F_q\subset U_\beta$ and we have the identifications
\begin{equation}\label{eq:chart}U_\alpha^\vee\cong H^1(F_p,\mathbb{K}^\ast),U_\beta^\vee\cong H^1(F_q,\mathbb{K}^\ast),\end{equation}
where $U_\alpha^\vee$ and $U_\beta^\vee$ are rigid analytic $T$-duals of $U_\alpha$ and $U_\beta$ respectively. So the birational maps $\mathit{\Upsilon}_{\alpha\beta}$ can be used to glue different charts $U_\alpha^\vee$ and $U_\beta^\vee$ together to obtain the corrected, completed SYZ mirror
\begin{equation}Y^\vee=\bigsqcup_{\alpha\in A}H^1(F_p,\mathbb{K}^\ast)/\sim,\end{equation}
where the equivalence relation $\sim$ identifies points which are mapped to each other under $\mathit{\Upsilon}_{\alpha\beta}$.\bigskip

The local construction of the family Floer module can be applied each chart $U_\alpha\subset Y$. In particular, for any Lagrangian submanifold $L\subset Y$ which is oriented, Spin, and tautologically unobstructed when restricted to $U_\alpha$, we obtain a rigid analytic subvariety
\begin{equation}\mathrm{supp}H^\ast\mathcal{F}(L_\alpha,\xi_\alpha)\subset U_\alpha^\vee,\end{equation}
where $L_\alpha=L|U_\alpha$ and $\xi_\alpha$ is the one induced from certain $\xi\in H^1(L,U_\mathbb{K})$.

Under the additional assumption that
\begin{equation}\label{eq:gluecon}\mathit{\Upsilon}_{\alpha\beta}\big(\mathrm{supp}H^\ast\mathcal{F}(L_\alpha,\xi_\alpha)\big)=\mathrm{supp}H^\ast\mathcal{F}(L_\beta,\xi_\beta),\end{equation}
the rigid analytic subvarieties coming from local constructions can be glued together to obtain a well-defined rigid analytic subvariety
\begin{equation}\mathrm{supp}H^\ast\mathcal{F}(L,\xi)\subset Y^\vee,\end{equation}
which is defined to be our mirror transformation of $(L,\xi)$.

For the examples considered in this paper, namely when $L$ is a regular fiber of $\pi_H$ or $\pi_L$, the gluing condition (\ref{eq:gluecon}) above is satisfied. This is due to the fact that in our case, the wall-crossing map $\mathit{\Upsilon}_{\alpha\beta}$ is the identity when restricted to  the subvarieties $\mathrm{supp}H^\ast\mathcal{F}(L_\alpha,\xi_\alpha)$ of $U_\alpha^\vee$. See Sections \ref{section:twin-tcy} and \ref{section:twin-blp} for details.

\section{Twin Lagrangian fibrations}\label{section:twin}

\subsection{Geometric setup}\label{section:twin-ell}
This subsection is essentially an overview of $\cite{ly}$. We explain the motivation of introducing the notion of a twin Lagrangian fibration (Definition \ref{definition:4.1}) and give some speculations of such a geometric structure. For simplicity, we consider here the mirror of an elliptic Calabi-Yau manifold. With some additional effort, one should be able to extend most of the considerations here to the general case of Calabi-Yau manifolds fibered by rigid analytic subvarieties.\bigskip

Let $Y^\vee$ be an $n$-dimensional Calabi-Yau manifold over $\mathbb{K}$, with $\rho:Y^\vee\rightarrow S$ being an elliptic fibration on $Y^\vee$. Suppose that we have a well-defined compactified relative Jacobian
\begin{equation}\mathcal{Y}^\vee:=\overline{\mathrm{Jac}}(Y^\vee/S),\end{equation}
which is the moduli space of semistable sheaves of rank one, degree zero supported on the fibers of $\rho$. Mirror symmetry predicts that such a moduli space can be identified with certain moduli space $\mathcal{Y}$ of Lagrangian branes $(L_\star,\xi_\star)$ with $\xi\in H^1(L,U_\mathbb{K})$ on the mirror symplectic manifold $Y$.\\
For simplicity, we impose the following additional assumptions:
\begin{itemize}
\item The Lagrangian submanifolds $L_\star$ are oriented, Spin and unobstructed.
\item The elliptic fibration $\rho:Y^\vee\rightarrow S$ has a section.
\item The singular fibers of $\rho$ only have nodal or cuspidal singularities.
\end{itemize}
Since the elliptic fibration $\rho$ may contain singular fibers, $L_\star$ may be singular as well. Here we assume that $L_\star$ is an \textit{immersed} Lagrangian submanifold, so that it has a well-defined Floer theory. $L_\star$ (decorated with local systems) in $\mathcal{Y}$ are disjoint (or disjoinable by Hamiltonian isotopies) from each other in view of the fact that
\begin{equation}\mathit{Ext}^\ast(\mathcal{O}_{F_1},\mathcal{O}_{F_2})=0\end{equation}
for two different fibers $F_1,F_2$ of $\rho$. Since $\dim_\mathbb{K}(\mathcal{Y}^\vee)=n$, we see that $\dim_\mathbb{R}(\mathcal{Y})=2n$, which implies that these Lagrangian submanifolds foliate at least an open subset $Y^\circ$ of $Y$. If one looks closer, the Lagrangians $L_\star$ should actually foliate the whole space $Y$. In fact, the expected isomorphism
\begin{equation}\mathit{HF}^\ast(L_\star,L_\star)\cong\mathit{Ext}^\ast(\mathcal{O}_F,\mathcal{O}_F)\end{equation}
suggests that the Lagrangian submanifolds $L_\star$ are tori, where $F$ is a fiber of $\rho$ so that $\mathcal{O}_F$ is mirror to $L_\star$. Treating the foliation on $Y^\circ$ by Lagrangian submanifolds $L_\star$'s as an SYZ fibration, its mirror $(Y^\circ)^\vee$ can be identified with the relative Jacobian $\mathcal{Y}^\vee$ of $Y^\vee$. By our assumptions on $\rho$, there is an isomorphism $\mathcal{Y}^\vee\cong Y^\vee$. This implies that $(Y^\circ)^\vee\cong Y^\vee$. Passing to the mirror side, the embedding $\iota:Y^\circ\hookrightarrow Y$ should then be a symplectomorphism into $\iota(Y^\circ)$. In conclusion, besides the putative SYZ fibration $\pi:Y\rightarrow B$, the mirror of an elliptic Calabi-Yau manifold carries another Lagrangian fibration $\pi_\star:Y\rightarrow B_\star$.\bigskip

Generalizing the above picture, it also makes sense to consider other fibrations on $Y^\vee$ by rigid analytic subvarieties, which should still induce Lagrangian fibrations on $Y$. However, in general, it could be difficult to make sense of the moduli space $\mathcal{Y}^\vee$. This is the case of the Lagrangian fibrations $\pi_H$ and $\pi_L$ defined above in Section \ref{section:lagop}, as these fibrations are induced from fibrations on $Y^\vee$ by non-complete subvarieties.\bigskip

The first assumption on the elliptic fibration $\rho$ in the heuristic arguments above is actually not necessary. This is illustrated in the following example, which is also considered in $\cite{ma}$.\bigskip

\textit{Example}. Consider $\mathbb{R}^4$ with coordinates $x_1,x_2,x_3,x_4$ and equipped with the standard symplectic form $\omega_{\mathbb{C}^2}$. Let $\mathit{\Gamma}'\subset\mathbb{R}^4$ be the lattice defined by translations of integral vectors in the directions $x_2,x_3,x_4$ and the transformation
\begin{equation}(x_1,x_2,x_3,x_4)\rightarrow(x_1+1,x_2,x_3,x_4+x_3).\end{equation}
The symplectic form $\omega_{\mathbb{C}^2}$ is invariant under the action of $\mathit{\Gamma}'$ and therefore descends to the quotient $M=\mathbb{R}^4/\mathit{\Gamma}'$. $M$ is called the \textit{Kodaira-Thurston manifold} $\cite{wt}$. In $\cite{is}$ it is noticed that on $M$ there are two inequivalent Lagrangian fibrations. The first one $\pi:M\rightarrow B$ is obtained by projecting to the coordinates $x_2$ and $x_3$. It has a Lagrangian section and we regard it as the SYZ fibration on $M$. The second fibration is a principal $T^2$-bundle $\pi_\star:M\rightarrow B_\star$, and therefore has \textit{no} Lagrangian section. $\pi_\star$ is obtained by projecting to the coordinates $x_1$ and $x_3$. One can see that the rigid analytic $T$-dual of $\pi:M\rightarrow B$ is a primary Kodaira surface $M^\vee$, which is a principal elliptic bundle $\rho:M^\vee\rightarrow E$ over an elliptic curve $E$. Since $\rho$ does not admit a section, the relative Jacobian $J$ is not isomorphic to $M^\vee$. In fact, $J\cong E\times E$ is an abelian surface. We remark that the Lagrangian fibration mirror to $\rho:M^\vee\rightarrow E$ is exactly the principal $T^2$-bundle $\pi_\star:M\rightarrow T^2$.\\
One can also consider the mirror of $\pi_\star$, which involves a gerbe $\alpha_M\in H^2(J,\mathcal{O}^\ast)$. This is in fact the obstruction to the existence of a relative Poincar\'{e} sheaf over $M^\vee\times_E J$ $\cite{ac1}$.\bigskip

Back to the general setting, we want to derive some further constraints on the second Lagrangian fibration $\pi_\star:Y\rightarrow B_\star$. Let $L$ be a fiber of $\pi:Y\rightarrow B$, which we also assume to be unobstructed. Then its mirror $\mathcal{O}_y$ is a skyscraper sheaf with $y\in Y^\vee$. Homological mirror symmetry predicts that
\begin{equation}\mathit{HF}^\ast(L,L_\star)\cong\mathit{Ext}^\ast(\mathcal{O}_y,\mathcal{O}_F)\end{equation}
as $\mathbb{K}$-vector spaces. Adopting the Morse-Bott model of Lagrangian Floer cohomology $\cite{fooo1}$, the above isomorphism suggests that the intersection $L\cap L_\star$ is clean and is a codimension 1 submanifold of $L$. (One may regard it as a simplifying assumption which does not violate our general prediction.) As a concrete example, take $L$ and $L_\star$ to be respectively the fibers of the two Lagrangian fibrations on $(T^4,\omega_\mathrm{std})$ defined by projecting respectively to $(x_2,x_4)$ and $(x_1,x_4)$, and let $\rho: E\times E\rightarrow E$ be the obvious elliptic fibration on the mirror.

\paragraph{Remark} In $\cite{ly}$, it is also required that the Lagrangian fibations $\pi$ and $\pi_\star$ should admit Lagrangian sections. This condition is not imposed here because we want to include the Kodaira-Thuston manifold $M$ as an example which admits a twin Lagrangian fibration.

\subsection{Twin Lagrangian fibrations on toric Calabi-Yau manifolds}\label{section:twin-tcy}
Let $\overline{X}$ be an $n$-dimensional toric Calabi-Yau manifold. In this subsection, we study the mirror symmetry of the twin Lagrangian fibration structure on $X=\overline{X}\setminus D$ given by the Lagrangian fibrations $\pi_G$ and $\pi_H$ constructed in Section $\ref{section:gghl}$.

As mentioned in the introduction, such a twin Lagrangian fibration is expected to be mirror to the affine conic bundle $p_0:X^\vee\rightarrow(\mathbb{K}^\ast)^{n-1}$ on the mirror. Based on the mirror construction described in Appendix \ref{section:syz-tcy} and the mirror transformation introduced in Section \ref{section:trans}, we now verify this. The following unobstructedness result is needed in our proof.

Recall from Appendix \ref{section:syz-tcy} that the base $B$ of the SYZ fibration $\pi_G$ is separated by the wall $\mathcal{W}\subset B$ into two chambers
\begin{equation}B_1=\{b_1>0\}\times\mathbb{R}^{n-1},B_2=\{b_1<0\}\times\mathbb{R}^{n-1}.\end{equation}
\begin{lemma}\label{lemma:unob}
Let $U_i=\pi_G^{-1}(B_i)$, where $i=1,2$. For any regular fiber $L_\star$ of $\pi_H$, $L_\star^i=L_\star|{U_i}$ is tautologically unobstructed as a Lagrangian submanifold in $U_i$.
\end{lemma}
\begin{proof}
Recall from Section \ref{section:gghl} that the set of critical points of $\pi_H$ is identical to that of $\pi_G$, so after removing the fibers of $\pi_G$ over the wall $\mathcal{W}:=\{0\}\times\mathbb{R}^{n-1}$ in $B$, the critical locus of $\pi_H$ has been removed. After restricting the definition of $\pi_H$ to $U_i$, it then becomes a Lagrangian $T^{n-1}\times\mathbb{R}$ bundle $f_i:U_i\rightarrow S^1\times\mathbb{R}^{n-1}$. Note that after removing the critical locus of $\pi_H$, a singular fiber splits into two copies of $T^{n-1}\times\mathbb{R}$, and each of them serves as a fiber of $f_i$. The lemma then follows from the isomorphism
\begin{equation}\pi_2(U_i,L_\star^i)\cong\pi_2(S^1\times\mathbb{R}^{n-1})=0.\end{equation}
\end{proof}
We now verify that the gluing condition (\ref{eq:gluecon}) is satisfied for the regular fibers of Lagrangian fibration $\pi_H$, so that $\mathrm{supp}H^\ast\mathcal{F}(L_\star,\xi_\star)$ is well-defined as a rigid analytic subvariety of $X^\vee$.
\begin{proposition}\label{proposition:glue}
The rigid analytic subvarieties $\mathrm{supp}H^\ast\mathcal{F}(L_\star^1,\xi_\star^1)\subset U_1^\vee$ and $\mathrm{supp}H^\ast\mathcal{F}(L_\star^2,\xi_\star^2)\subset U_2^\vee$ patch together under the wall-crossing map $\mathit{\Upsilon}_{12}:U_1^\vee\dashrightarrow U_2^\vee$ to produce a rigid analytic subvariety $\mathrm{supp}H^\ast\mathcal{F}(L_\star,\xi_\star)\subset X^\vee$.
\end{proposition}
\begin{proof}
From (\ref{eq:chart}), we see that $U_1^\vee\cong U_2^\vee\cong(\mathbb{K}^\ast)^n$. Denote the coordinates on $U_1^\vee$ and $U_2^\vee$ by $(x_1,\cdot\cdot\cdot,x_{n-1},z)$ and $(x_1',\cdot\cdot\cdot,x_{n-1}',y)$ respectively. Let $L_b=\pi_G^{-1}(b)$ be a regular fiber of the SYZ fibration $\pi_G$, which is equipped with the unitary rank one local system $\xi_b\in H^1(L_b,U_\mathbb{K})$. By the SYZ mirror constructions recalled in Section \ref{section:trans} or Appendix \ref{section:syz-tcy}, we know that the Lagrangian brane $(L_b,\xi_b)$ determines a point of the mirror $X^\vee$. In order to write down the coordinates of this point, fix a reference fiber $L_\mathrm{ref}$ of $\pi_G$. By (\ref{eq:mircoo}) and (\ref{eq:mircoo2}) from Appendix \ref{section:syz-tcy}, we have
\begin{equation}x_j=x_j'=T^{\int_{\mathit{\Theta_j}}\omega_{\overline{X}}}\xi_b(\theta_j),j=1,\cdot\cdot\cdot,n-1,\end{equation}
where $\{\theta_j\}$ is a set of generators of $H_1(L_b,\mathbb{Z})$ which span a regular orbit of the Hamiltonian $T_\mathbb{R}(N_\nu)$-action, and $\{\mathit{\Theta}_j\}$ are cylinders traced out by these loops with their boundaries lying on $L_b$ and the reference fiber $L_\mathrm{ref}$. In particular, this implies that the wall-crossing map $\mathit{\Upsilon}_{12}:U_1^\vee\dashrightarrow U_2^\vee$ is the identity for the first $n-1$ coordinates.

We claim that for $i=1,2$, the rigid analytic subvarieties $\mathrm{supp}H^\ast\mathcal{F}(L_\star^i,\xi_\star^i)\subset U_i^\vee$ are defined by the linear equations
\begin{equation}
x_1=s_1,\cdot\cdot\cdot,x_{n-1}=s_{n-1}
\end{equation}
and
\begin{equation}
x_1'=s_1,\cdot\cdot\cdot,x_{n-1}'=s_{n-1},
\end{equation}
where $s_i\in\mathbb{K}^\ast$ are fixed constants independent of $i$, so in particular they can be patched together under $\mathit{\Upsilon}_{12}$ to the affine conic in $\mathbb{K}^2$ defined by
\begin{equation}
yz=g(s_1,\cdot\cdot\cdot,s_{n-1}).
\end{equation}

To see this, let $L_\star$ be a regular fiber of $\pi_H$. By Proposition \ref{proposition:inter}, its restriction $L_\star^i\subset U_i$ fibers as a $T^{n-1}$ bundle over the submanifold $Q_i\subset B_i$ defined by
\begin{equation}Q_i=\big\{(b_1,\mathbf{b}_2)\in B_i|\mathbf{b}_2=C_4\big\},\end{equation}
where $b_1$ and $\mathbf{b}_2$ are respectively standard coordinates on the $\mathbb{R}$ and $\mathbb{R}^{n-1}$ factors, and $C_4\in\mathbb{R}^{n-1}$ is a constant vector. By Lemma \ref{lemma:unob}, $L_\star^i\subset U_i$ is a tautologically unobstructed Lagrangian submanifold, so in particular the coherent sheaves $H^\ast\mathcal{F}(L_\star^i,\xi_\star^i)$ are well-defined for any choice of $\xi_\star\in H^1(L_\star,U_\mathbb{K})$, where $\xi_\star^i$ is the restriction of $\xi_\star$ to $L_\star^i$.

To determine $\mathrm{supp}H^\ast\mathcal{F}(L_\star,\xi_\star)$, we need to consider the Floer cohomology groups
\begin{equation}\label{eq:Floer}\mathit{HF}^\ast\big((L_b,\xi_b),(L_\star^i,\xi_\star^i)\big)\end{equation}
for every $b\in B_i$. Since $L_b\cap L_\star^i\neq\emptyset$ precisely when $b\in Q_i$, by the isotopy invariance of Floer cohomology, we see that
\begin{equation}\mathit{HF}^\ast\big((L_b,\xi_b),(L_\star^i,\xi_\star^i)\big)\neq0\end{equation}
only if $b\in Q_i$. Since the Lagrangian submanifolds $L_b$ and $L_\star^i$ intersect cleanly for any $b\in Q_i$, and both of $L_b$ and $L_\star^i$ are tautologically unobstructed, we have
\begin{equation}\mathit{HF}^\ast\big((L_b,\xi_b),(L_\star^i,\xi_\star^i)\big)=H^\ast\big(L_b\cap L_\star^i,(\xi_b-\xi_\star^i)|(L_b\cap L_\star^i)\big),\end{equation}
where the right hand side is ordinary cohomology with local coefficients, from which we deduce $\mathit{HF}^\ast\big((L_b,\xi_b),(L_\star^i,\xi_\star^i)\big)\neq0$ if and only if $\xi_b=\xi_\star^i$ in $H^1(L_b\cap L_\star^i,U_\mathbb{K})$.

The non-vanishing conditions of the Floer cohomology (\ref{eq:Floer}) is equivalent to requiring that
\begin{itemize}
	\item $T^{\int_{\mathit{\Theta_j}}\omega_{\overline{X}}}$ remains constant,
	\item $\xi_b(\theta_j)=\xi_\star(\theta_j)$,
\end{itemize}
for $1\leq j\leq n-1$, where the first condition above follows from the invariance of symplectic area in a relative homotopy class with Lagrangian boundary condition, namely $\pi_2(U_i,L_\star^i)$. These two conditions together imply the invariance of the coordinates $x_j$ and $x_j'$, which completes the proof.
\end{proof}

We have proved:
\begin{theorem}\label{theorem:tcy}
The twin Lagrangian fibration $B\xleftarrow{\pi_G}X\xrightarrow{\pi_H}B_\star$ on $X$ is induced from the affine conic bundle structure $p_0:X^\vee\rightarrow(\mathbb{K}^\ast)^{n-1}$ on its mirror $X^\vee$ in the sense of Definition \ref{definition:twin-ind}, namely
\begin{equation}\mathrm{supp}H^\ast\mathcal{F}(L_\star,\xi_\star)=p_0^{-1}(s_1,\cdot\cdot\cdot,s_{n-1})\end{equation}
for any regular fiber $L_\star$ of $\pi_H$ equipped with any $\xi_\star\in H^1(L_\star,U_\mathbb{K})$. 
\end{theorem}

\subsection{Twin Lagrangian fibrations on blowups of toric varieties}\label{section:twin-blp}

We have a parallel story for $\overline{X}^\vee$. Recall that two Lagrangian fibrations $\pi_A$ and $\pi_L$ on $X^\vee=\overline{X}^\vee\setminus D^\vee$ have been described in Section $\ref{section:lag-blp}$.

As we have mentioned in the introduction, the twin Lagrangian fibration $B^\vee\xleftarrow{\pi_A}X^\vee\xrightarrow{\pi_L}B_\star^\vee$ on $X^\vee$ is expected to be mirror to the fibration $w_0:X\rightarrow\mathbb{K}^\ast$. Recall that up to an additive constant, $w_0$ is the defining function of $K_{\overline{X}}$, so it has a unique singular fiber and the regular fibers of $w_0$ are isomorphic to $(\mathbb{K}^\ast)^{n-1}$.

Recall from Appendix \ref{section:converse} that the base $B^\vee$ of $\pi_L$ is separated by the wall $\mathcal{W}_\blacklozenge=\Delta\times\mathbb{R}_{>-\varepsilon}$ into chambers $B_\alpha^\vee$ parametrized by the finite set $A$, i.e.
\begin{equation}B^\vee\setminus\mathcal{W}_\blacklozenge=\bigsqcup_{\alpha\in A}B_\alpha^\vee.\end{equation}
\begin{lemma}\label{lemma:unob1}
The Lagrangian submanifold $L_\star^\alpha=L_\star|U_\alpha^\vee$ with $\alpha\in A$ is tautologically unobstructed in $U_\alpha^\vee$, where $U_\alpha^\vee=\pi_A^{-1}(B_\alpha^\vee)$.
\end{lemma}
\begin{proof}
Recall from Section \ref{section:lag-blp} that we have an identification between the set of critical points of the two Lagrangian fibrations $\pi_A$ and $\pi_L$, therefore after removing the fibers of $\pi_A$ over the wall $\mathcal{W}_\blacklozenge\subset B^\vee$, we have removed the critical locus of $\pi_L$ as well. Taking the inverse image of $\pi_A$ over the chamber $B_\alpha^\vee$, we get an open subset $U_\alpha^\vee\subset V_0$, without loss of generality we may assume
\begin{equation}
U_\alpha^\vee=\{a_i<|x_i|<b_i,1\leq i\leq n-1\}\subset(\mathbb{C}^\ast)^{n-1}.
\end{equation}
The fibration $\pi_L$, when restricted to $U_\alpha^\vee$, then become a Lagrangian $\mathbb{R}^{n-1}\times S^1$ bundle $f_\alpha:U_\alpha^\vee\rightarrow T^{n-1}\times\mathbb{R}_{>-\varepsilon}$. The lemma then follows easily from the isomorphism
\begin{equation}\pi_2(U_\alpha^\vee,L_\star^\alpha)\cong\pi_2(T^{n-1}\times\mathbb{R}_{>-\varepsilon})=0.\end{equation}
\end{proof}
As in the toric Calabi-Yau case, we need to verify here that the gluing condition (\ref{eq:gluecon}) holds for regular fibers of the Lagrangian fibration $\pi_L$.
\begin{proposition}
The rigid analytic subvarieties $\mathrm{supp}H^\ast\mathcal{F}(L_\star^\alpha,\xi_\star^\alpha)\subset U_\alpha$, $\alpha\in A$, can be patched together under the wall-crossing map $\mathit{\Upsilon}_{\alpha\beta}:U_\alpha\dashrightarrow U_\beta$ to produce a rigid analytic subvariety $\mathrm{supp}H^\ast\mathcal{F}(L_\star,\xi_\star)\subset X$.
\end{proposition}
\begin{proof}
The proof is analogous to that of Proposition \ref{proposition:glue}. By (\ref{eq:chart}), we have $U_\alpha\cong(\mathbb{K}^\ast)^n$ for each $\alpha$. Denote the coordinates on $U_\alpha$ by $(v_{\alpha,1},\cdot\cdot\cdot,v_{\alpha,{n-1}},w_{\alpha,0})$. Let $L_b=\pi_A^{-1}(b)$ be a regular fiber of the SYZ fibration $\pi_A$ equipped with the unitary rank one local system $\xi_b\in H^1(L_b,U_\mathbb{K})$. By the mirror constructions recalled in Section \ref{section:trans} or Appendix \ref{section:converse}, we know that the Lagrangian brane $(L_b,\xi_b)$ determines a point of the mirror $X$. To write down the coordinates of this point, fix a reference fiber $L_\mathrm{ref}$ of $\pi_A$. By (\ref{eq:sfcalp}) from Appendix \ref{section:converse}, we have
\begin{equation}w_{\alpha,0}=T^{\int_{\mathit{\Theta}_0}\omega_\varepsilon}\xi_b(\theta_0)\end{equation}
for every $\alpha\in A$, where $\theta_0\in H^1(L_b,\mathbb{Z})$ corresponds to the $S^1$-orbit of the Hamiltonian action $(\ref{eq:ciract})$ and it traces out the cylinder $\mathit{\Theta}_0$ under the isotopy from $L_b$ to the reference fiber $L_\mathrm{ref}$. In particular, this implies that all the wall-crossing maps $\mathit{\Upsilon}_{\alpha\beta}:U_\alpha\dashrightarrow U_\beta$ are the identity for the last coordinate.

We claim that for $\alpha\in A$, the rigid analytic subvariety $\mathrm{supp}H^\ast\mathcal{F}(L_\star^\alpha,\xi_\star^\alpha)\subset U_\alpha$ is defined by the linear equation
\begin{equation}
w_{\alpha,0}=s
\end{equation}
where $s\in\mathbb{K}^\ast$ is a fixed constant independent of $\alpha$, so in particular they can be patched together under $\mathit{\Upsilon}_{\alpha\beta}$ to the hypersurface $w_0^{-1}(s)\subset X$.

To see this, let $L_\star$ be a regular fiber of $\pi_L$. By Proposition \ref{proposition:clean}, its restriction $L_\star^\alpha$ in $U_\alpha^\vee$ fibers as an $S^1$ bundle over the submanifold $Q_\alpha\subset B_\alpha^\vee$ defined by
\begin{equation}Q_\alpha=\big\{(\mathbf{b}_1,b_2)\in B^\vee_\alpha|b_2=C_8\big\},\end{equation}
where $\mathbf{b}_1$ is the standard coordinate on $\mathbb{R}^{n-1}$ and $C_8\in\mathbb{R}_{>-\varepsilon}$ is some fixed constant. By Lemma \ref{lemma:unob1}, $L_\star^\alpha\subset U_\alpha^\vee$ is a tautologically unobstructed Lagrangian submanifold, so in particular the coherent sheaves $H^\ast\mathcal{F}(L_\star^\alpha,\xi_\star^\alpha)$ are well-defined for any choice of the local system $\xi_\star\in H^1(L_\star,U_\mathbb{K})$, where $\xi_\star^\alpha$ denotes the restriction of $\xi_\star$ to $L_\star^\alpha$.

In order to to determine $\mathrm{supp}H^\ast\mathcal{F}(L_\star,\xi_\star)$, we have to study the Floer cohomology groups
\begin{equation}\label{eq:Floer1}\mathit{HF}^\ast\big((L_b,\xi_b),(L_\star^\alpha,\xi_\star^\alpha)\big)\end{equation}
for every $b\in B_\alpha^\vee$. Since $L_b\cap L_\star^\alpha\neq\emptyset$ precisely when $b\in Q_\alpha$, we see that
\begin{equation}\mathit{HF}^\ast\big((L_b,\xi_b),(L_\star^\alpha,\xi_\star^\alpha)\big)\neq0\end{equation}
only when $b\in Q_\alpha$. Since the Lagrangian submanifolds $L_b$ and $L_\star^\alpha$ intersect over points of $Q_\alpha$, and both of $L_b$ and $L_\star^\alpha$ are tautologically unobstructed as Lagrangian submanifolds in $U_\alpha^\vee$, we have
\begin{equation}\mathit{HF}^\ast\big((L_b,\xi_b),(L_\star^\alpha,\xi_\star^\alpha)\big)=H^\ast\big(L_b\cap L_\star^\alpha,(\xi_b-\xi_\star^\alpha)|(L_b\cap L_\star^\alpha)\big),\end{equation}
where the right hand side is ordinary cohomology with local coefficients. This shows that $\mathit{HF}^\ast\big((L_b,\xi_b),(L_\star^\alpha,\xi_\star^\alpha)\big)\neq0$ if and only if $\xi_b=\xi_\star^\alpha$ in $H^1(L_b\cap L_\star^\alpha,U_\mathbb{K})$.

The non-vanishing conditions of the Floer cohomology (\ref{eq:Floer1}) can be equally expressed as
\begin{itemize}
	\item $T^{\int_{\mathit{\Theta_0}}\omega_\varepsilon}$ remains constant,
	\item $\xi_b(\theta_0)=\xi_\star(\theta_0)$,
\end{itemize}
where the first condition above follows from the invariance of symplectic area inside a relative homotopy class with Lagrangian boundary conditions, namely $\pi_2(X^\vee,L_\star^\alpha)$. These two conditions together imply the invariance of the coordinate $w_{\alpha,0}$, which completes the proof.
\end{proof}

We summarize our main result in this subsection in the following theorem.
\begin{theorem}\label{theorem:blp}
Let $\overline{X}^\vee$ be the blow up of $V\times\mathbb{C}$ along $H\times0$, where $V$ is a toric variety satisfying Assumption \ref{assumption:3.1}, and $H\subset V$ is a nearly tropical hypersurface. Then the twin Lagrangian fibration $B^\vee\xleftarrow{\pi_A}X^\vee\xrightarrow{\pi_L}B_\star^\vee$ on $X^\vee$ is induced from the fibration $w_0:X\rightarrow\mathbb{K}^\ast$ on the mirror manifold $X$, in the sense that
\begin{equation}\mathrm{supp}H^\ast\mathcal{F}(L_\star,\xi_\star)=w_0^{-1}(s)\end{equation}
for any regular fiber $L_\star$ of $\pi_L$ equipped with any $\xi_\star\in H^1(L_\star,U_\mathbb{K})$.
\end{theorem}

Since $w_0:X\rightarrow\mathbb{K}^\ast$ has a unique singular fiber, by carefully examining the definition of coordinates on the mirror, we have a somehow stronger conclusion.
\begin{proposition}\label{proposition:refine}
For a generic choice of the regular fiber $L_\star$ of $\pi_L$ and any $\xi_\star\in H^1(L_\star,U_\mathbb{K})$, $\mathrm{supp}H^\ast\mathcal{F}(L_\star,\xi_\star)$ is a regular fiber of $w_0:X\rightarrow\mathbb{K}^\ast$.
\end{proposition}
\begin{proof}
Fix a chamber $B_\alpha^\vee$ to work with, since $w_0$ has a unique singular fiber $w_0^{-1}(-T^\varepsilon)=D_{\overline{X}}$, we only need to show that the mirror coordinate $w_{\alpha,0}$ avoids the value $-T^\varepsilon\in\mathbb{K}^\ast$. But by the definition of $w_{\alpha,0}$ in Appendix \ref{section:converse}, this forces the symplectic area $\int_{\mathit{\Theta}_0}\omega_{\varepsilon}$ to be some constant $\lambda_0$. It is known that up to a multiplicative constant, we have
\begin{equation}\int_{\mathit{\Theta}_0}\omega_{\varepsilon}=C_8-\lambda_\mathrm{ref},\end{equation}
where $\lambda_\mathrm{ref}$ is the second coordinate of the reference fiber $L_\mathrm{ref}$ of $\pi_A$, see $\cite{aak}$. So $w_{\alpha,0}=-T^\varepsilon$ forces $C_8=\lambda_1$ for some suitable constant $\lambda_1$. But choosing the fiber $L_\star$ so that $C_8\neq\lambda_1$ does not affect the genericity of the choice of $L_\star$.
\end{proof}

\section{Applications in four dimensions}\label{section:sing}
This section contains two simple applications of twin Lagrangian fibrations studied in the last section, which are inspired respectively by Section 4.4 of the paper of Smith $\cite{is}$ and the work of Lekili-Maydanskiy $\cite{lm}$. To get a more explicit picture, we restrict ourselves here to the case of symplectic 4-manifolds.

\subsection{Homological mirror symmetry for $\mathit{Bl}_K(\mathbb{C}^2)$}\label{section:fuk}
Let $V=\mathbb{C}$ in the setting of Section \ref{section:lag-blp}. Consider a partial compactification of our space $X^\vee$, which is simply a smoothing of the $A_{p-1}$ singularity
\begin{equation}X_{p-1}^\vee=\big\{(x,y,z)\in\mathbb{C}^3|yz=(x-r_1)\cdot\cdot\cdot(x-r_p)\big\}.\end{equation}
We equip $X_{p-1}^\vee$ with the restriction of the constant symplectic form on $\mathbb{C}^3$, making it into a Liouville manifold. The nearly tropical condition on $H\subset\mathbb{C}$ is now equivalent to
\begin{equation}|r_1|\ll\cdot\cdot\cdot\ll|r_p|.\end{equation}
By projecting to $x$, we get an exact Lefschetz fibration
\begin{equation}p_0:X_{p-1}^\vee\rightarrow\mathbb{C},\end{equation}
whose regular fibers are symplectomorphic to $T^\ast S^1$. On the other hand, by the discussions above, there is a twin Lagrangian fibration on the open dense subset $X^\vee\subset X_{p-1}^\vee$ formed by $\pi_A$ and $\pi_L$. Strictly speaking, here the symplectic structure on $X^\vee$ differs from the general case treated before, but the Lagrangian fibrations $\pi_A$ and $\pi_L$ still exist, and can actually be explicitly written down as
\begin{equation}\pi_A=\left(\log|x|,|y|^2/2-|z|^2/2\right),\pi_L=\left(\arg(x),|y|^2/2-|z|^2/2\right).\end{equation}
In particular, we see that both of the fibrations $\pi_A$ and $\pi_L$ are smooth. We shall always work in the ``generic case", namely when every singular fiber of $\pi_L$ contains a unique singularity, which requires an additional assumption on the positions of the $r_i$'s. After assuming this, any singular fiber $\pi_L^{-1}(\bullet)$ of $\pi_L$ is then a union of two Lagrangian discs $L^+$ and $L^-$, meeting transversely at the singularity of $\pi_L^{-1}(\bullet)$.\\
We make the following simple observation.
\begin{lemma}\label{lemma:thimdisc}
The Lagrangian discs $L^\pm$ in any singular fiber of $\pi_L$ are Lefschetz thimbles of $p_0:X_{p-1}^\vee\rightarrow\mathbb{C}$.
\end{lemma}
\begin{proof}
Under the map $p_0$, the fibers of $\pi_L$ are projected to rays on the $x$ coordinate plane. The singular fibers of $\pi_L$ correspond precisely to those rays passing through the critical values of $p_0$, namely $r_1,\cdot\cdot\cdot,r_p\in\mathbb{C}$. Over these points, the orbits of the Hamiltonian $S^1$-action $(\ref{eq:ciract})$ on $X^\vee$ degenerate to points. The Lagrangian discs $L^\pm$ are then projected by $p_0$ to two vanishing paths $\gamma_i^+$ and $\gamma_i^-$ which meet at a unique critical value $r_i$, which shows that they are thimbles meeting transversely at the critical point in the fiber $p_0^{-1}(r_i)$.
\end{proof}
Since $H^1(X_{p-1}^\vee)=0$, up to quasi-isomorphism there is a well-defined $\mathbb{Z}$-graded directed $A_\infty$ category $\mathcal{F}uk(p_0)$ associated to the Lefschetz fibration $p_0:X_{p-1}^\vee\rightarrow\mathbb{C}$, see $\cite{ps}$. The objects of $\mathcal{F}uk(p_0)$ are closed unobstructed Lagrangian submanifolds equipped with gradings and Spin structures as well as Lefschetz thimbles, and the morphisms between thimbles are defined using Hamiltonian diffeomorphisms with constant slopes near infinity. More explicitly, denote by $\mathit{\Delta}_1,\cdot\cdot\cdot,\mathit{\Delta}_p$ a basis of Lefschetz thimbles of $p_0$ associated to the vanishing paths $\gamma_1,\cdot\cdot\cdot,\gamma_p$ which are straight lines, and assume
\begin{equation}\arg(\gamma_1)<\cdot\cdot\cdot<\arg(\gamma_p),\end{equation}
then we have
\begin{equation}\mathit{HF}^\ast(\mathit{\Delta}_i,\mathit{\Delta}_j)=\left\{\begin{array}{ll}H^\ast(S^1,\mathbb{K}) & i<j,\\ \mathbb{K}e_{\mathit{\Delta}_i} & i=j,\\ 0 & i>j.\end{array}\right.\end{equation}
By Lemma $\ref{lemma:thimdisc}$, we can choose the basis of thimbles $\mathit{\Delta}_1,\cdot\cdot\cdot,\mathit{\Delta}_p$ so that they are Lagrangian discs contained in the singular fibers of $\pi_L$. For every $i\in\{1,\cdot\cdot\cdot,p\}$, we choose the copy of $L^\pm\subset\pi_L^{-1}(\bullet)$ so that its vanishing path $\gamma_i$ is an outward pointing ray starts at $r_i$, and still denote the resulting basis of Lefschetz thimbles by $\mathit{\Delta}_1,\cdot\cdot\cdot,\mathit{\Delta}_p$.

Define $\mathcal{F}uk(\pi_L)$ to be the Fukaya category consisting of $\mathit{\Delta}_1,\cdot\cdot\cdot,\mathit{\Delta}_p$ as its objects with the morphism spaces setting to be
\begin{equation}\mathit{CF}^\ast(\mathit{\Delta}_i,\mathit{\Delta}_j)=\left\{\begin{array}{ll}\mathbb{K} & i=j,\\ 0 & i\neq j,\end{array}\right.\end{equation}
which means for thimbles disjoint from each other, we don't use Hamiltonian perturbations to create intersections between them. It's easy to see with such a definition, we actually have an equivalence
\begin{equation}\mathcal{F}uk\big(\{p\textrm{ }\mathrm{pts}\}\big)\cong\mathcal{F}uk(\pi_L),\end{equation}
where $\mathcal{F}uk\big(\{p\textrm{ }\mathrm{pts}\}\big)$ is the Fukaya category of $p$ distinct points.\bigskip

Mirror to the Lagrangian fibration $\pi_L$ is the fibration $w_0:X\rightarrow\mathbb{K}^\ast$, which admits a unique singular fiber $w^{-1}(0)=D_{\overline{X}}$. For this fibration, Orlov defined the \textit{triangulated category of singularities} $\cite{do}$:
\begin{equation}\label{eq:dmf}D^b_\mathrm{sing}\big(w^{-1}(0)\big)=D^b\mathit{Coh}\big(w^{-1}(0)\big)/\mathrm{Perf}\big(w^{-1}(0)\big),\end{equation}
where $\mathrm{Perf}\big(w^{-1}(0)\big)$ denotes the full triangulated subcategory of perfect complexes in the derived category of coherent sheaves $D^b\mathit{Coh}\big(w^{-1}(0)\big)$, and the right hand side is a Verdier quotient. Although by definition the category $D^b_\mathrm{sing}\big(w^{-1}(0)\big)$ is naturally $\mathbb{Z}_2$-graded, in our situation we can lift it to a $\mathbb{Z}$-grading by specifying a $\mathbb{K}^\ast$-action on $X$ so that
\begin{itemize}
\item $w$ has weight 2,
\item $-1\in\mathbb{K}^\ast$ acts trivially.
\end{itemize}
On each $\mathbb{K}^2$ chart of the toric variety $\overline{X}$ with coordinates $x$ and $y$, this can be seen explicitly by letting $\mathbb{K}^\ast$ act with weight 0 on $x$ and weight 2 on $y$.\\
Denote by $D^\pi_\mathrm{sing}\big(w^{-1}(0)\big)$ the idempotent completion of $D^b_\mathrm{sing}\big(w^{-1}(0)\big)$, which is still triangulated by $\cite{bs}$.
\begin{proposition}\label{proposition:hmsfib}
There is an equivalence
\begin{equation}\Phi_\bullet:D\mathcal{F}uk(\pi_L)\cong D_\mathrm{sing}^\pi\big(w^{-1}(0)\big)\end{equation}
between $\mathbb{Z}$-graded triangulated categories, where $D\mathcal{F}uk(\pi_L)=H^0\big(\mathit{Tw}\mathcal{F}uk(\pi_L)\big)$ denotes the derived Fukaya category of $\mathcal{F}uk(\pi_L)$.
\end{proposition}
\begin{proof}
This is an application of the simplest example of Kn\"{o}rrer periodicity, which states that
\begin{equation}D^\pi_\mathrm{sing}\big(w^{-1}(0)\big)\cong D^b\mathit{Coh}(\mathit{pt})\end{equation}
when $p=1$, where $D^b\mathit{Coh}(\mathit{pt})$ is the derived category of coherent sheaves of one point. However, by Proposition 1.14 of $\cite{do}$, such an equivalence extends to the case of any $p\in\mathbb{Z}_{\geq0}$. This in particular shows that the total morphism algebra of $D^\pi_\mathrm{sing}\big(w^{-1}(0)\big)$ is isomorphic to the semi-simple ring $\mathbb{K}^p$, which proves the desired equivalence. The functor $\Phi_\bullet$ is defined in the obvious way. In particular, it sends $\mathit{\Delta}_i$ to an idempotent of the skyscraper sheaf $\mathcal{O}_{s_i}$ of the singularity $s_i$ of $w^{-1}(0)$ on object level.
\end{proof}
To get a deeper understanding of the above equivalence, we further (partially) compactify $X^\vee_{p-1}$ to $\overline{X}^\vee$, which is just $\mathit{Bl}_K(\mathbb{C}^2)$, where $K\subset\mathbb{C}^2$ is a finite set consisting of $p$ distinct points. For later purposes we further assume that these $p$ points lie on a sufficiently large circle $\widetilde{C}$ in the $x$-coordinate plane centered at the origin so that \begin{equation}\mathrm{dist}(r_i,r_j)>2\sqrt{\pi}\end{equation}
for any two different points $r_i,r_j\in K$. Note that assuming $K\subset\widetilde{C}$ will make $H=K\subset\mathbb{C}$ fail to be nearly tropical, which then results in a singular mirror by $\cite{aak}$. However, this can be avoided easily if one is willing to take more care about the positions of the $r_i$'s, we assume this here only to simplify our expositions. By blowing up with equal amounts at every point of $K$, one then obtains a (non-compact) monotone symplectic manifold $\mathit{Bl}_K(\mathbb{C}^2)$. Following Smith $\cite{is1}$, we consider the Lagrangian correspondences
\begin{equation}\mathbb{L}_1=\bigsqcup_{i=1}^p S_\mathrm{eq}^1\subset\{p\textrm{ }\mathrm{pts}\}\times\mathbb{P}^1\cong E,\mathbb{L}_2=\partial\nu_E\subset E^-\times\mathit{Bl}_K(\mathbb{C}^2),\end{equation}
where $E\subset\mathit{Bl}_K(\mathbb{C}^2)$ denotes the exceptional divisor, which is a disjoint union of $\mathbb{P}^1$'s, and $E^-$ is the symplectic manifold $(E,-\omega_E)$. $S_\mathrm{eq}^1\subset\mathbb{P}^1$ denotes the equator. $\nu_E$ is a tubular neighborhood of the exceptional divisor, and $\partial\nu_E$ denotes its boundary.
\begin{lemma}
Equip every $\mathbb{P}^1$ in $E$ with the symplectic form $2\omega_{\mathrm{FS}}$, where $\omega_{\mathrm{FS}}$ denotes the Fubini-Study form, and equip $\mathit{Bl}_K(\mathbb{C}^2)$ with the symplectic form $\Omega_1$ so that every exceptional curve has area $\pi$. Then the Lagrangian correspondence $\mathbb{L}_2\subset E^-\times\mathit{Bl}_K(\mathbb{C}^2)$ associated to $\partial\nu_E$ is monotone.
\end{lemma}
\begin{proof}
In our case, $\partial\nu_E$ is a disjoint union of coisotropic $S^3$'s in $\mathit{Bl}_K(\mathbb{C}^2)$. The symplectic form $\Omega_1$ is well-defined on $\mathit{Bl}_K(\mathbb{C}^2)$ because by our assumption there are balls of radius strictly larger than $\sqrt{\pi}$ centered at the $p$ points in $K\subset\mathbb{C}^2$, and these balls are disjoint from each other. Since $\pi_1(\mathbb{L}_2)=0$ as a Lagrangian submanifold in $E^-\times\mathit{Bl}_K(\mathbb{C}^2)$, we only need to show the symplectic manifold $E^-\times\mathit{Bl}_K(\mathbb{C}^2)$ is \textit{spherically monotone}, namely the homomorphisms
\begin{equation}c_1,\omega:\pi_2\big(E^-\times\mathit{Bl}_K(\mathbb{C}^2)\big)\rightarrow\mathbb{R}\end{equation}
defined by the first Chern class and the symplectic form are positively proportional. In our case, $E^-\times\mathit{Bl}_K(\mathbb{C}^2)$ is is equipped with the product symplectic form $-2\omega_\mathrm{FS}\times\Omega_1$. On the other hand,
\begin{equation}c_1\big(E^-\times\mathit{Bl}_K(\mathbb{C}^2)\big)\in H^2(E^-)\oplus H^2\big(\mathit{Bl}_K(\mathbb{C}^2)\big)\end{equation}
is easily computed to be $(-2,\cdot\cdot\cdot,-2,1,\cdot\cdot\cdot,1)$, which shows that $\mathbb{L}_2$ is monotone.
\end{proof}
In what follows, we shall always equip $\mathit{Bl}_K(\mathbb{C}^2)$ with the monotone symplectic form $\Omega_1$. Since $\mathbb{L}_1$ is obviously monotone, the geometric composition $\mathbb{L}_1\circ\mathbb{L}_2$ defines a monotone Lagrangian submanifold in $\mathit{Bl}_K(\mathbb{C}^2)$, which is a disjoint union of $p$ monotone Lagrangian tori $\widetilde{T}_1,\cdot\cdot\cdot,\widetilde{T}_p$. These tori can also be seen using the Lefschetz fibration on $\mathit{Bl}_K(\mathbb{C}^2)$. Start with the trivial fibration $\mathbb{C}^2\rightarrow\mathbb{C}$ by projecting to one of the two coordinates, after blowing up at $K\subset\mathbb{C}^2$ we get a Lefschetz fibration
\begin{equation}\widetilde{p}_0:\mathit{Bl}_K(\mathbb{C}^2)\rightarrow\mathbb{C}\end{equation}
whose vanishing cycles are homotopically trivial.\\
Under the map $\widetilde{p}_0$, the tori $\widetilde{T}_1,\cdot\cdot\cdot,\widetilde{T}_p$ project to disjoint circles with the same radius $\widetilde{c}_1,\cdot\cdot\cdot,\widetilde{c}_p$ centered at the critical values of $\widetilde{p}_0$, namely $\widetilde{r}_1,\cdot\cdot\cdot,\widetilde{r}_p\in\widetilde{C}$. Denote by $D_i\subset\mathbb{C}$ the closed disc bounded by $\widetilde{c}_i$, and by $\widetilde{p}_0:V_i\rightarrow D_i$ the restriction of the Lefschetz fibration to $V_i=\widetilde{p}_0^{-1}(D_i)$. Note that one can arrange the ordering so that $r_i=\widetilde{r}_i$, so there is a 1-1 correspondence between the Lefschetz thimbles of $p_0$ and $\widetilde{p}_0$. We shall adopt this convention from now on.\\
Denote by $\widetilde{T}$ the monotone Lagrangian torus lying over the interior corner of the moment polytope of the toric variety $\mathcal{O}_{\mathbb{P}^1}(-1)$.
\begin{lemma}\label{lemma:modeldisc}
There is a diffeomorphism between moduli spaces
\begin{equation}\label{eq:modulid}\mathcal{M}_1^{\mathcal{O}(-1)}(\widetilde{T},\beta)\cong\mathcal{M}_1^{\mathit{Bl}_K(\mathbb{C}^2)}(\widetilde{T}_i,\beta_K)\end{equation}
as compact manifolds, where $\mathcal{M}_1^{\mathcal{O}(-1)}(\widetilde{T},\beta)$ denotes the moduli space of stable holomorphic discs with one boundary marked point represented by the the class $\beta\in\pi_2\big(\mathcal{O}_{\mathbb{P}^1}(-1),\widetilde{T}\big)$ with respect to the standard complex structure, and $\beta_K\in\pi_2\big(\mathit{Bl}_K(\mathbb{C}^2),\widetilde{T}_i\big)$ is the corresponding class of holomorphic discs bounded by $\widetilde{T}_i$.
\end{lemma}
\begin{proof}
First consider the Lefschetz fibration $\widetilde{p}_0:\mathit{Bl}_K(\mathbb{C}^2)\rightarrow\mathbb{C}$, by applying maximum principle to the holomorphic function $\widetilde{p}_0\circ u:\mathbb{D}\rightarrow D_i$, we see that for every holomorphic disc $u:(\mathbb{D},\partial\mathbb{D})\rightarrow\big(\mathit{Bl}_K(\mathbb{C}^2),\widetilde{T}_i\big)$, $\mathrm{im}(u)\subset V_i$.

Since the Lagrangian torus $\widetilde{T}_i$ is monotone, only disc bubblings and sphere bubblings are possible. We then show that for every holomorphic disc $u:(\mathbb{D},\partial\mathbb{D})\rightarrow\big(V_i,\widetilde{T}_i\big)$, neither disc bubbling nor sphere bubbling can occur. For disc bubbles, first notice that for dimension reasons, $\widetilde{T}_i\subset V_i$ only bounds stable discs of Maslov index 2, in other words, their classes $\beta_K\in\pi_2\big(V_i,\widetilde{T}_i\big)$ must have intersection number 1 with
\begin{equation}Z_i:=\widetilde{p}_0^{-1}(\widetilde{r}_i)\cup(\widetilde{V}\cap V_i),\end{equation}
where $\widetilde{V}$ is the proper transform of $V\times\{0\}\cong\mathbb{C}$. Because of this, any disc bubble must have Maslov index 0, i.e. it's a holomorphic disc $u$ with $\mathrm{im}(u)\cap Z_i=\emptyset$. But $V_i\setminus Z_i$ can be realized as a Lagrangian torus bundle over some aspherical manifold with $\widetilde{T}_i$ as one of its fiber, which implies that $\pi_2(V_i\setminus Z_i,\widetilde{T}_i)=0$ and there is no disc bubbling. Sphere bubbles can be excluded simply by noticing that there is only one holomorphic sphere of Chern number 1 in $V_i$, which is an exceptional curve, and other holomorphic spheres in $\mathit{Bl}_K(\mathbb{C}^2)$ are disjoint from $V_i$. This proves that
\begin{equation}\mathcal{M}_1^{\mathit{Bl}_K(\mathbb{C}^2)}(\widetilde{T}_i,\beta_K)\cong\mathcal{M}_1^{V_i}(\widetilde{T}_i,\beta_K)\end{equation}
and their compactness.

Its easy to see the above arguments can also be applied to $\mathcal{O}_{\mathbb{P}^1}(-1)$, from which we get the desired diffeomorphism $(\ref{eq:modulid})$.
\end{proof}
The following result is proved in $\cite{is1}$ for $\widetilde{T}\subset\big(\mathcal{O}_{\mathbb{P}^1}(-1),\Omega_1\big)$.
\begin{corollary}\label{corollary:clif}
For any $1\leq i\leq p$, equip $\widetilde{T}_i$ with the Spin structure which is non-trivial on both $S^1$ factors and the $\mathbb{Z}_2$-grading coming from orientation, then
\begin{equation}\mathit{HF}^\ast(\widetilde{T}_i,\widetilde{T}_i)\cong\mathrm{Cl}_2,\end{equation}
as $\mathbb{Z}_2$-graded algebras, where $\mathrm{Cl}_2$ is the Clifford algebra associated to a non-degenerate quadratic form on $\mathbb{K}^2$.
\end{corollary}
\begin{proof}
Lemma $\ref{lemma:modeldisc}$ allows us to identify the enumeration of holomorphic discs bounded by $\widetilde{T}_i\subset\mathit{Bl}_K(\mathbb{C}^2)$ with that of $\widetilde{T}\subset\mathcal{O}_{\mathbb{P}^1}(-1)$. The latter one is studied in detail in $\cite{da2}$ and $\cite{chc}$, from which we get the toric fiber $\widetilde{T}$ bounds three families of holomorphic discs, one for each toric boundary divisor in $\mathcal{O}_{\mathbb{P}^1}(-1)$. By $(\ref{eq:pot}$) we see that the superpotentials for $\widetilde{T}_i$ are given by
\begin{equation}W(\widetilde{T}_i)=z_1+z_2+T^{-1/2} z_1z_2,\end{equation}
where $z_i\in\mathbb{K}^\ast$. Since by maximum principle, all the holomorphic strips bounded by $\widetilde{T}_i$ are local, the fact that the unique critical point of $W(\widetilde{T}_i)$ is non-degenerate shows that $\mathit{HF}^\ast(\widetilde{T}_i,\widetilde{T}_i)\cong\mathrm{Cl}_2$.
\end{proof}
\begin{proposition}\label{proposition:generation}
$\widetilde{T}_1,\cdot\cdot\cdot,\widetilde{T}_p$ split-generate the non-zero eigensummand $\mathcal{F}uk\big(\mathit{Bl}_K(\mathbb{C}^2)\big)_{\widetilde{\lambda}}$ of the monotone Fukaya category, where
\begin{equation}\widetilde{\lambda}e_{\widetilde{T}_i}=\mathfrak{m}_0(\widetilde{T}_i)=m_0(\widetilde{T}_i)[\widetilde{T}_i]\end{equation}
is defined in terms of enumeration of Maslov index 2 holomorphic discs, see Appendix \ref{section:syz}.
\end{proposition}
\begin{proof}
Since $\mathit{Bl}_K(\mathbb{C}^2)$ is a symplectic manifold conical at infinity, we can consider the open-closed string maps
\begin{equation}\mathit{OC}^0:\mathit{HF}^\ast(\widetilde{T}_i,\widetilde{T}_i)\rightarrow\mathit{QH}^\ast\big(\mathit{Bl}_K(\mathbb{C}^2)\big)\end{equation}
constructed in $\cite{rs}$. By Corollary \ref{corollary:clif}, $[pt]$ defines a cocycle in $\mathit{HF}^\ast(\widetilde{T}_i,\widetilde{T}_i)$. It follows that $\mathit{OC}^0([pt])\neq0$. On the other hand, by a version of the Cardy relation (see $\cite{ps5}$), the images of cocycles for different $\widetilde{T}_i$'s are orthogonal to each other with respect to the quantum intersection pairing on $\mathit{QH}^\ast\big(\mathit{Bl}_K(\mathbb{C}^2)\big)$. From explicit computations one can see that \begin{equation}\mathit{QH}^\ast\big(\mathit{Bl}_K(\mathbb{C}^2)\big)/\mathit{QH}^\ast\big(\mathit{Bl}_K(\mathbb{C}^2)\big)_0\cong\bigoplus_{i=1}^p\mathbb{K}\end{equation}
is semisimple, where $\mathit{QH}^\ast\big(\mathit{Bl}_K(\mathbb{C}^2)\big)_0$ is the nilpotent summand with respect to the quantum multiplication of $c_1\big(\mathit{Bl}_K(\mathbb{C}^2)\big)$.

Denote by $\mathcal{C}$ the full subcategory formed by $\widetilde{T}_1,\cdot\cdot\cdot,\widetilde{T}_p$, from the above we see that the open-closed map $\mathit{OC}$ restricted on $\mathcal{C}$ hits an invertible element of $\mathit{QH}^\ast\big(\mathit{Bl}_K(\mathbb{C}^2)\big)_{\widetilde{\lambda}}$. By the generation criterion obtained in $\cite{rs}$, the claim follows.
\end{proof}
A similar generation result holds for a more general class of non-compact monotone symplectic manifolds, see $\cite{yl}$.
\bigskip

For simplicity, we shall omit the subscript and simply write $\mathcal{F}uk\big(\mathit{Bl}_K(\mathbb{C}^2)\big)$. On the other hand, from $\cite{is1}$ we see that the Lagrangian tori $\widetilde{T}_i$ have non-trivial idempotents, which we denote by $e_i^\pm$, and $e_i^+\cong e_i^-[1]$ in the split-closure of the category of twisted complexes $\mathit{Tw}^\pi\mathcal{F}uk\big(\mathit{Bl}_K(\mathbb{C}^2)\big)$. Denote by $\widetilde{\mathit{\Delta}}_1,\cdot,\cdot\cdot,\widetilde{\mathit{\Delta}}_p$ the Lefschetz thimbles of $\widetilde{p}_0$ so that their vanishing paths $\widetilde{\gamma}_1,\cdot\cdot\cdot,\widetilde{\gamma}_p$ point outwards along the radical directions.
\begin{lemma}\label{lemma:idem}
$\widetilde{\mathit{\Delta}}_i$ is the unique Lefschetz thimble of $\widetilde{p}_0$ which has non-trivial intersections with the monotone Lagrangian torus $\widetilde{T}_i$, and
\begin{equation}\mathit{HF}^\ast(\widetilde{T}_i,\widetilde{\mathit{\Delta}}_i)\cong H^\ast(S^1,\mathbb{K}).\end{equation}
\end{lemma}
\begin{proof}
The first half of the lemma follows from the positions of vanishing paths associated to the thimbles $\widetilde{\mathit{\Delta}}_i$ we choose. Corollary $\ref{corollary:clif}$ and the same proof as in $\cite{is1}$, Lemma 4.26 yield the second half.
\end{proof}
By $\cite{mww}$, the composite monotone Lagrangian correspondence $\mathbb{L}_1\circ\mathbb{L}_2$ defines a functor
\begin{equation}\widehat{\Phi}_{\mathbb{L}_1\circ\mathbb{L}_2}:\mathcal{F}uk\big(\{p\textrm{ }\mathrm{pts}\}\big)\rightarrow\mathcal{F}uk\big(\mathit{Bl}_K(\mathbb{C}^2)\big)\end{equation}
This functor has an idempotent summand
\begin{equation}\widehat{\Phi}(e^+):\mathcal{F}uk\big(\{p\textrm{ }\mathrm{pts}\}\big)\rightarrow\mathit{Tw}^\pi\mathcal{F}uk\big(\mathit{Bl}_K(\mathbb{C}^2)\big),\end{equation}
see $\cite{is1}$. Identifying $\mathcal{F}uk\big(\{p\textrm{ }\mathrm{pts}\}\big)$ with $\mathcal{F}uk(\pi_L)$ as $\mathbb{Z}_2$-graded categories in the obvious way, we obtain a functor from $\mathcal{F}uk(\pi_L)$ to $\mathit{Tw}^\pi\mathcal{F}uk\big(\mathit{Bl}_K(\mathbb{C}^2)\big)$. By slight abuse of notation, this functor is still denoted $\widehat{\Phi}(e^+)$. Note that on the level of objects, $\widehat{\Phi}(e^+)$ sends $\mathit{\Delta}_i$ to $e_i^+$.\\
In view of Proposition $\ref{proposition:hmsfib}$, there is a similar phenomenon on the mirror side, namely a fully faithful embedding
\begin{equation}\Phi(e^+)^\vee:D_\mathrm{sing}^\pi\big(w^{-1}(0)\big)\hookrightarrow D^\pi(X,W)\end{equation}
defined by regarding the generators on the left hand side as idempotents in $D^b(X,W)$, where $D^\pi(X,W)$ is the triangulated category consisting $D_\mathrm{sing}^\pi\big(W(\widetilde{T}_i)^{-1}(-T^{1/2})\big)$ as its direct summands, or it can be regarded as a triangulated category of D-branes of type B in the sense of $\cite{do}$. Here $w:X\rightarrow\mathbb{K}$ and $W(\widetilde{T}_i)$ are defined on $(\mathbb{K}^\ast)^2$ or $\mathbb{K}^2$, both categories are equipped with their natural $\mathbb{Z}_2$-gradings. On the object level, $\Phi(e^+)^\vee$ sends an idempotent of $\mathcal{O}_{s_i}$ to an idempotent of $\mathcal{O}_{t_i}$, where $t_i\in W(\widetilde{T}_i)^{-1}(-T^{1/2})$ is the unique singularity.\bigskip

\textbf{Remark.} Note that the expressions of $W(\widetilde{T}_i)$ coincide with the discussions in Appendix \ref{section:converse}, see Lemmas \ref{lemma:3.7} and \ref{lemma:3.8}. By restricting $W=w_0+w_1+w_2$ defined on the toric Calabi-Yau surface $\overline{X}$ to each $\mathbb{K}^2$ coordinate chart, we get essentially the superpotential $W(\widetilde{T}_i)$ associated to the monotone Lagrangian torus $\widetilde{T}_i$. This justifies the notation $D^\pi(X,W)$ we use here.\bigskip

To summarize our discussions in this subsection, we need a slight variation of the localized mirror functor $\widehat{\Phi}_\mathrm{CHL}$ introduced in $\cite{chl}$ and $\cite{chl1}$, which will be made precise in Appendix \ref{section:localized}. This is an $A_\infty$ functor
\begin{equation}\label{eq:localmf}\widehat{\Phi}_\mathrm{CHL}:\mathcal{F}uk\big(\mathit{Bl}_K(\mathbb{C}^2)\big)\rightarrow\mathit{MF}(X,W)\end{equation}
where
\begin{equation}\mathit{MF}(X,W):=\bigsqcup_{i=1}^p\mathit{MF}\big(W(\widetilde{T}_i)\big)\end{equation}
is the disjoint union of the category of matrix factorizations $\mathit{MF}\big(W(\widetilde{T}_i)\big)$. On the other hand, it is proved by Orlov in $\cite{do}$ and $\cite{do3}$ that there is an equivalence
\begin{equation}\mathit{\Sigma}:H^0\big(\mathit{MF}(X,W)\big)\rightarrow D^b(X,W)\end{equation}
between $\mathbb{Z}_2$-graded triangulated categories. Denote by
\begin{equation}\label{eq:derived}\Phi_\mathrm{CHL}:D\mathcal{F}uk\big(\mathit{Bl}_K(\mathbb{C}^2)\big)\rightarrow D^b(X,W)\end{equation}
the composition of the induced functor of $\widehat{\Phi}_\mathrm{CHL}$ on derived categories with $\mathit{\Sigma}$. $\Phi_\mathrm{CHL}$ can then be extended to a functor on split-closures of both sides in $(\ref{eq:derived})$, which by abuse of notation we still denote by $\Phi_\mathrm{CHL}$.
\begin{theorem}\label{theorem:idemcomm}
The following diagram is commutative.
\begin{equation}\label{eq:idemcomm}
\xymatrix{
D\mathcal{F}uk(\pi_L) \ar[d]_{\Phi_\bullet} \ar[r]^-{\Phi(e^+)}
 &D^\pi\mathcal{F}uk\big(\mathit{Bl}_K(\mathbb{C}^2)\big)\ar[d]^{\Phi_\mathrm{CHL}}\\
D^\pi_\mathrm{sing}\big(w^{-1}(0)\big) \ar[r]^-{\Phi(e^+)^\vee}
 &D^\pi(X,W)}
\end{equation}
where $\Phi(e^+)$ is the functor induced by $\widehat{\Phi}(e^+)$ on split-closed derived categories.
\end{theorem}
\begin{proof}
First observe that all four functors in the diagram $(\ref{eq:idemcomm})$ are equivalences. This is proved for $\Phi_\bullet$ in Proposition $\ref{proposition:hmsfib}$. The functor $\Phi(e^+)$ is by definition fully faithful. Since the Lagrangian torus $\widetilde{T}_i\subset\mathit{Bl}_K(\mathbb{C}^2)$ is generated by the thimbles $\widetilde{\mathit{\Delta}}_1,\cdot\cdot\cdot,\widetilde{\mathit{\Delta}}_p$ by Proposition 5.8 of $\cite{is}$ and $\widetilde{T}_i$ only has non-empty intersections with the thimble $\widetilde{\mathit{\Delta}}_i$ by Lemma \ref{lemma:idem}, we see that $\widetilde{T}_i$ is in fact generated by $\widetilde{\mathit{\Delta}}_i$. This shows that $\Phi(e^+)$ is an equivalence. Similar reasonings show that $\Phi(e^+)^\vee$ is an equivalence. For any Lagrangian torus $\widetilde{T}_i$ in $\mathcal{F}uk\big(\mathit{Bl}_K(\mathbb{C}^2)\big)$, its Floer cohomology $\mathit{HF}^\ast(\widetilde{T}_i,\widetilde{T}_i)$ has been computed in Corollary $\ref{corollary:clif}$. To show that $\Phi_\mathrm{CHL}$ is an equivalence, we need to compute $\hom_{D^\pi_\mathrm{sing}}(\mathcal{O}_{t_i},\mathcal{O}_{t_i})$ in $D_\mathrm{sing}^\pi\big(W(\widetilde{T}_i)^{-1}(-T^{1/2})\big)$. Since $t_i$ is a nodal singularity, the computation can be done in the standard local model when $\mathcal{O}_{t_i}$ is the skyscraper sheaf at the origin in $\mathrm{Spec}\big(\mathbb{K}[x,y]/xy\big)$. The result is
\begin{equation}\hom_{D^\pi_\mathrm{sing}}(\mathcal{O}_{t_i},\mathcal{O}_{t_i})\cong\mathit{HF}^\ast(\widetilde{T}_i,\widetilde{T}_i)\end{equation}
as $\mathbb{Z}_2$-graded $\mathbb{K}$-vector spaces, which proves that $\Phi_\mathrm{CHL}$ is indeed an equivalence by Theorem 1.3 of $\cite{chl1}$.

By the above, it's enough to show that the diagram $(\ref{eq:idemcomm})$ commutes on the object level, since multiplying by $\mathbb{K}^\ast$ is isomorphic to an identity functor on the field $\mathbb{K}$ viewed as a linear category with a single object. But the commutativity of $(\ref{eq:idemcomm})$ on the level of objects is straightforward from definitions.
\end{proof}
The above result gives a geometric interpretation of the homological mirror symmetry equivalence $D^\pi\mathcal{F}uk\big(\mathit{Bl}_K(\mathbb{C}^2)\big)\cong D^\pi(X,W)$ as the mirror symmetry between the fibration structures $\pi_L$ and $w_0$ in the framework of twin Lagrangian fibrations.

\subsection{Non-displaceable Lagrangian tori in rational homology balls}\label{section:ga}
As a geometric application of the structure of twin Lagrangian fibrations, we consider here the rational homology balls $B_{p,q}$ studied in $\cite{lm}$. These are Stein surfaces defined as the quotients of the $A_{p-1}$ Milnor fiber $X_{p-1}^\vee$ by certain finite group actions. More precisely, we assume throughout this subsection that $r_j=e^{\frac{2j\pi}{p}i}\in\mathbb{C}$ are $p$-th roots of unity in the definition of $X^\vee$. Denote by $G_{p,q}$ the following free action of $\mathbb{Z}_p$ on $X_{p-1}^\vee$:
\begin{equation}\xi\cdot(x,y,z)=(\xi^qx,\xi y,\xi^{-1}z),\end{equation}
where $\xi\in\mathbb{Z}_p$ is any primitive $p$-th root of unity and $(p,q)=1$. Then the quotient $B_{p,q}=X_{p-1}^\vee/G_{p,q}$ is a rational homology ball for any $p\geq2$. It is proved in $\cite{lm}$ that there is no closed exact Lagrangian submanifolds in $B_{p,q}$ for any $p>2$ and $\mathit{SH}^\ast(B_{p,q})\neq0$, therefore these Stein surfaces provide non-trivial examples to test the non-vanishing criterion of symplectic cohomology due to Seidel and Smith for 4-dimensional Liouville manifolds, see Section 5 of $\cite{ps3}$.\\
In this subsection, we give a new proof of the non-vanishing result $\mathit{SH}^\ast(B_{p,q})\neq0$ from the point of view of twin Lagrangian fibrations. We begin with the following observation.
\begin{proposition}
After the quotient by $G_{p,q}$, the twin Lagrangian fibration on $X^\vee$ descends to a twin Lagrangian fibration on $B_{p,q}$ away from some divisor $D_{p,q}^\vee\subset B_{p,q}$.
\end{proposition}
\begin{proof}
Under the action of $G_{p,q}$, the fibers of $\pi_A$ are invariant, which implies that it descends to a Lagrangian torus fibration $\pi_A^{p,q}$ on $B_{p,q}\setminus D_{p,q}^\vee$, where $D_{p,q}^\vee\subset B_{p,q}$ is the divisor which lifts to the conic $\{yz=1\}\subset X_{p-1}^\vee$. On the other hand, $G_{p,q}$ acts on the fibers of $\pi_L$ freely, which implies that the orbits of the fibers of $\pi_L$ under the action of $G_{p,q}$ form another Lagrangian $\mathbb{R}\times S^1$ fibration $\pi_L^{p,q}$ on $B_{p,q}\setminus D_{p,q}^\vee$. Since the action of $G_{p,q}$ preserves the moment map of the Hamiltonian $S^1$-action on $X^\vee$, the intersections between the fibers of $\pi_A^{p,q}$ and $\pi_L^{p,q}$ are clean and affine.
\end{proof}
The idea of our proof goes as follows. Pick a circle $c_p\subset\mathbb{C}$ in the base of the Lefschetz fibration $p_0$ on $X_{p-1}^\vee$ centered at the origin, so that the radius of $c_p$ is larger than 1. The vanishing cycles fibered over $c_p$ then defines a Lagrangian torus $T_p\subset X_{p-1}^\vee$, which is a fiber of $\pi_A$. From the point of view of family Floer cohomology, the non-displaceability of $T_p$ should be detected with the help of the additional Lagrangian fibration $\pi_L$. More precisely, the Floer cohomologies $\mathit{HF}^\ast\big(T_p,\pi_L^{-1}(b)\big)$ with $b\in B_\star^\vee$ should lead to a spectral sequence converging to $\mathit{HF}^\ast(T_p,T_p)$.\\
In our case, the situation is much simpler because Lemma $\ref{lemma:thimdisc}$ provides a set of distinguished Lagrangians coming from the singular fibers $\pi_L$, namely the Lefschetz thimbles $\mathit{\Delta}_1,\cdot\cdot\cdot,\mathit{\Delta}_p$ of $p_0$ which form a full exceptional collection in $D\mathcal{F}uk(p_0)$, so it suffices to look at finitely many Floer cohomology groups $\mathit{HF}^\ast(T_p,\mathit{\Delta}_i)$. In view of the proposition above, we then expect that the same story will descend to the quotient $B_{p,q}$. By a theorem of Seidel-Smith $\cite{ps3}$, this then implies the non-triviality of $\mathit{SH}^\ast(B_{p,q})$. In short, the non-vanishing of $\mathit{SH}^\ast(B_{p,q})$ should follow by studying the Lagrangian Floer theory between certain fibers of $\pi_A^{p,q}$ and $\pi_L^{p,q}$ which form a twin Lagrangian fibration on $B_{p,q}\setminus D_{p,q}^\vee$.
\begin{lemma}\label{lemma:torthim}
For every $1\leq i\leq p$, $\mathit{HF}^\ast(T_p,\mathit{\Delta}_i)\cong H^\ast(S^1,\mathbb{K})$.
\end{lemma}
\begin{proof}
The Lagrangian submanifolds $T_p$ and $\mathit{\Delta}_i$ meet cleanly along a circle. The easiest way to see the vanishing of the Morse-Bott-Floer differentials on $C_\ast(T_p\cap\mathit{\Delta}_i)$ is to consider the Fukaya-Seidel category $\mathcal{F}uk(p_0)$. It is proved in Proposition 5.8 of $\cite{is1}$ that $T_p$ is generated by the thimbles $\mathit{\Delta}_1,\cdot\cdot\cdot,\mathit{\Delta}_p$ over any field $\mathbb{K}$ with $\mathrm{char}(\mathbb{K})\neq2$. On the other hand, by Corollary 9.1 of $\cite{aak}$, equipping $T_p$ with an appropriate Spin structure we have
\begin{equation}\mathit{HF}^\ast(T_p,T_p)\cong H^\ast(T^2,\mathbb{K})\end{equation}
as an algebra. Note that although it is assumed in Corollary 9.1 of $\cite{aak}$ that $|r_i|$'s are not the same, the chart $U_\alpha^\vee$ containing $T_p$ is not affected. More precisely, the holomorphic discs bounded by $T_p\subset X_{p-1}^\vee$ are sections of the Lefschetz fibration $p_0:X_{p-1}^\vee\rightarrow\mathbb{C}$ over the disc bounded by $c_p$. This disc counting can be broken into standard local models using the gluing results of $\cite{ps6}$ (see also $\cite{ww2}$ where the orientation issue is considered), which shows that the superpotential $W(T_p)$ is unaffected as long as the $r_i$'s lie inside the disc bounded by $c_p$. This implies that $\mathit{HF}^\ast(T_p,\mathit{\Delta}_i)\neq0$ for some $i$. But such a non-vanishing then extends to all $i$ by applying Dehn twists $\tau_{V_k}$ along the Lagrangian matching spheres $V_k\simeq S^2$ associated to the basic paths, i.e. linear paths in $\mathbb{C}$ connecting two $p$-th roots of unity.
\end{proof}
We now take the $G_{p,q}$-action into consideration. Note that although the Lagrangian fibrations $\pi_A$ and $\pi_L$ are compatible with the $G_{p,q}$-action, the Lefschetz fibration $p_0$ is not. More precisely, since $T_p$ is invariant under the action of $G_{p,q}$, it descends to a Lagrangian torus $T_{p,q}\subset B_{p,q}$. On the other hand, $G_{p,q}$ acts by permuting the thimbles $\mathit{\Delta}_i$, so after the quotient, we get a Lagrangian disc $\mathit{\Delta}_{p,q}\subset B_{p,q}$, which lifts to the union of the thimbles $\bigsqcup_{i=1}^p\mathit{\Delta}_i$ in $X_{p-1}^\vee$. However, $\mathit{\Delta}_{p,q}$ is no longer a Lefschetz thimble. Because of this, we consider the equivariant Fukaya category $\mathcal{F}uk(\pi_L)^{G_{p,q}}$, where $\mathcal{F}uk(\pi_L)$ is the subcategory of $\mathcal{F}uk(p_0)$ introduced in Section $\ref{section:fuk}$, but here we also include the torus fibers of $\pi_A$ as its objects. By Lemma \ref{lemma:genthim}, this will not cause any confusion when passing to derived categories. $\mathcal{F}uk(\pi_L)^{G_{p,q}}$ will serve as a technical replacement of the Fukaya category of the weakly unobstructed closed Lagrangian submanifolds and the Lagrangian disc $\mathit{\Delta}_{p,q}$ in $B_{p,q}$.

The free action of a finite group on the Fukaya category of closed Lagrangians is studied systematically in $\cite{www}$, which is essentially enough for our purposes here.
\begin{lemma}
The $A_\infty$ category $\mathcal{F}uk(\pi_L)^{G_{p,q}}$ is well-defined, with its objects the $G_{p,q}$-Lagrangian branes $G_{p,q}L$ equipped with $\mathbb{Z}_2$-gradings and Spin structures induced from $L\in\mathrm{Ob}\big(\mathcal{F}uk(\pi_L)\big)$, and morphisms
\begin{equation}\mathit{CF}^\ast_{G_{p,q}}(G_{p,q}L_0,G_{p,q}L_1)=\mathit{CF}^\ast(G_{p,q}L_0,G_{p,q}L_1)^{G_{p,q}}=\bigoplus_{g_0,g_1}\mathit{CF}^\ast(g_0L_0,g_1L_1)^{G_{p,q}},\end{equation}
where $g_i\in G_{p,q}/G_{p,q,L_i}$, with $G_{p,q,L_i}$ denotes the isotropy subgroup at $L_i$.
\end{lemma}
\begin{proof}
Since the regular fibers of $\pi_A$ are all invariant under the $G_{p,q}$-action, the gradings, Spin structures and $A_\infty$ structures of these Lagrangians are well-defined. The key point is that we can choose $G_{p,q}$-invariant Floer data and transversality of the corresponding moduli problems can be achieved with the $G_{p,q}$-invariant Floer data. The $G_{p,q}$-invariance of the Floer datum then implies that the $A_\infty$ compositions
\begin{equation}\mu_{G_{p,q}}^d:\mathit{CF}_{G_{p,q}}^\ast(G_{p,q}L_{d-1},G_{p,q}L_d)\otimes\cdot\cdot\cdot\otimes \mathit{CF}_{G_{p,q}}^\ast(G_{p,q}L_0,G_{p,q}L_1)\rightarrow\mathit{CF}^\ast(G_{p,q}L_0,G_{p,q}L_d)\end{equation}
lie in the $G_{p,q}$-invariant part of $\mathit{CF}^\ast(G_{p,q}L_0,G_{p,q}L_d)$.

For the thimbles $\mathit{\Delta}_i$, we only need to use a small perturbation at infinity to make their self-morphisms well-defined, so adding these objects to the category will not affect the proof in $\cite{www}$ that transversality for the relevant moduli spaces can be achieved within $G_{p,q}$-invariant Floer data. Assumption 4.1 of $\cite{www}$ is clearly satisfied since every thimble $\mathit{\Delta}_i$ is a copy of the Lagrangian disc $\mathit{\Delta}_{p,q}$ in the finite cover $\bigsqcup_{i=1}^p\mathit{\Delta}_i\rightarrow\mathit{\Delta}_{p,q}$, so $G_{p,q}\mathit{\Delta}_i=\bigsqcup_{i=1}^p\mathit{\Delta}_i$ inherits its grading and Spin structure from $\mathit{\Delta}_i$.
\end{proof}
In order to use $\mathcal{F}uk(\pi_L)^{G_{p,q}}$ to compute the Floer cohomology of $T_{p,q}\subset B_{p,q}$, the following simple generation result is useful.
\begin{lemma}\label{lemma:genthim}
In the equivariant Fukaya category $\mathcal{F}uk(\pi_L)^{G_{p,q}}$, the $G_{p,q}$-invariant monotone Lagrangian torus $T_p$ is generated by $G_{p,q}\mathit{\Delta}_i$ for any $i=1,\cdot\cdot\cdot,p$.
\end{lemma}
\begin{proof}
By Proposition 5.8 of $\cite{is1}$, in $\mathcal{F}uk(p_0)$, the monotone Lagrangian torus $T_p\subset X^\vee_{p-1}$ is generated by the Lefschetz thimbles $\mathit{\Delta}_1,\cdot\cdot\cdot,\mathit{\Delta}_p$ in the sense that
\begin{equation}T_{\mathit{\Delta}_1}\cdot\cdot\cdot T_{\mathit{\Delta}_p}(T_p)\cong 0\end{equation}
in $H^0\big(\mathcal{F}uk(p_0)\big)$, where $T_{\mathit{\Delta}_i}$ is the abstract twist along $\mathit{\Delta}_i$. In fact, since the monotone toric fiber $\widetilde{T}\subset\mathcal{O}_{\mathbb{P}^1}(-1)$ is generated by the unique Lefschetz thimble $\widetilde{\mathit{\Delta}}$ of $\widetilde{p}_0:\mathcal{O}_{\mathbb{P}^1}(-1)\rightarrow\mathbb{C}$, by Lemma \ref{lemma:torthim} the same functorial relation carries over to the situation here and gives
\begin{equation}T_{\mathit{\Delta}_i}(T_p)\cong0,i=1,\cdot\cdot\cdot,p\end{equation}
in $H^0\big(\mathcal{F}uk(p_0)\big)$, which then implies
\begin{equation}T_{\mathit{\Delta}_1,\cdot\cdot\cdot,\mathit{\Delta}_p}T_p\cong0,\end{equation}
where $T_{\mathit{\Delta}_1,\cdot\cdot\cdot,\mathit{\Delta}_p}$ is the generalized twist operation defined in $\cite{ps}$, Remark 5.1. This clearly descends to the relation
\begin{equation}T_{G_{p,q}\mathit{\Delta}_i}G_{p,q}T_p\cong0\end{equation}
in $H^0\big(\mathcal{F}uk(\pi_L)^{G_{p,q}}\big)$, which in particular yields the desired generation result.
\end{proof}
Recall from $\cite{www}$ that there is an $A_\infty$ functor
\begin{equation}\mathcal{T}:\mathcal{F}uk(B_{p,q})\rightarrow\mathcal{F}uk(X_{p-1}^\vee)^{G_{p,q}},\end{equation}
which is fully faithful and sends $L\subset B_{p,q}$ to its lift in $X_{p-1}^\vee$ on the object level. Again we consider here only Lagrangian torus fibers of $\pi_A^{p,q}$ and $\pi_A$ as objects of $\mathcal{F}uk(B_{p,q})$ and $\mathcal{F}uk(X_{p-1}^\vee)$ respectively.
\begin{corollary}[Lekili-Maydanskiy $\cite{lm}$]
The Stein surfaces $B_{p,q}$ are non-empty, i.e. $\mathit{SH}^\ast(B_{p,q})\neq0$. Furthermore, $\mathit{HF}^\ast(T_{p,q},T_{p,q})\cong H^\ast(T^2,\mathbb{K})$.
\end{corollary}
\begin{proof}
Recall that for an $A_\infty$ category $\mathscr{A}$ whose derived category $D(\mathscr{A})$ admits a full exceptional collection $Y_1,\cdot\cdot\cdot,Y_m$, the Beilinson spectral sequence with starting page
\begin{equation}E_1^{rs}=\big(\hom_{H(\mathscr{A})}(X_0,Y_{r+1}^!)\otimes\hom_{H(\mathscr{A})}(Y_{m-r},X_1)\big)^{r+s}\end{equation}
converges to $\hom_{H(\mathscr{A})}(X_0,X_1)$, where $Y_{r+1}^!$ denotes the Koszul dual of $Y_{m-r}$, see $\cite{fss,ps}$.\\
By Lemma \ref{lemma:genthim}, we can apply the above spectral sequence to the full subcategory $\mathscr{A}\subset\mathcal{F}uk(\pi_L)^{G_{p,q}}$ formed by the objects $G_{p,q}\mathit{\Delta}_i$ and $G_{p,q}T_p$. In our case, $m=1$, $Y_1=G_{p,q}\mathit{\Delta}_i$, and $X_0=X_1=G_{p,q}T_p$. It follows that the spectral sequence degenerates at $E_1$ page. This combines with Lemma \ref{lemma:torthim} shows that
\begin{equation}\mathit{HF}^\ast_{G_{p,q}}(T_p,T_p)\cong H^\ast(T^2,\mathbb{K}).\end{equation}
Now the transfer functor $\mathcal{T}$ gives us a fully faithful embedding $\mathcal{F}uk(B_{p,q})\hookrightarrow\mathcal{F}uk(\pi_L)^{G_{p,q}}$, which implies that
\begin{equation}\mathit{HF}^\ast(T_{p,q},T_{p,q})\cong\mathit{HF}^\ast_{G_{p,q}}(T_p,T_p).\end{equation}
\end{proof}

\appendix
\section{SYZ mirror constructions}\label{section:syz}

\subsection{SYZ mirror of toric Calabi-Yau manifolds}\label{section:syz-tcy}
This section is a summary of the work of Chan-Lau-Leung $\cite{cll}$ on SYZ mirror symmetry for toric Calabi-Yau manifolds, see also $\cite{da1}$ and Section 8 of $\cite{aak}$. Let $\overline{X}$ be an $n$-dimensional toric Calabi-Yau manifold. In this case, we use the Gross fibration $\pi_G:X\rightarrow B$ as the SYZ fibration. Denote by $\Delta\subset B$ the discriminant locus of $\pi_G$, then $\Delta\subset\mathcal{W}$ lies entirely in the unique wall $\mathcal{W}\subset B$ defined by $\{0\}\times\mathbb{R}^{n-1}$. $B\setminus\mathcal{W}$ consists of two chambers $B_1$ and $B_2$.
\begin{lemma}[$\cite{aak}$, Lemma 8.1]\label{lemma:obstruction}
A regular fiber $L_b\subset X$ of $\pi_G$ bounds some non-constant stable discs of Maslov 0 if and only of $L_b=\pi^{-1}_G(b)$ with $b\in \mathcal{W}\setminus\Delta$.
\end{lemma}
By the above lemma, the regular fibers over $B\setminus\mathcal{W}$ are tautologically unobstructed; while the other regular fibers are potentially obstructed, namely they bound holomorphic discs of Maslov index 0.\bigskip

Fix a reference fiber $L_\mathrm{ref}$ of $\pi_G$, and choose a basis $\theta_1,\cdot\cdot\cdot,\theta_{n-1},\theta_0^{(1)}$ of $H_1(L_\mathrm{ref},\mathbb{Z})$, with $-\theta_1,\cdot\cdot\cdot,-\theta_{n-1}$ corresponding to the factors of the Hamiltonian $T_\mathbb{R}(N_\nu)$-action defined in Section $\ref{section:gghl}$ and $-\theta_0^{(1)}$ corresponding to the last $S^1$ factor of $T_\mathbb{R}(N)$. There is an exact Lagrangian isotopy between $L_\mathrm{ref}$ and a toric fiber $\mu_{\overline{X}}^{-1}(\bullet)$.\\
Let $U_1\subset X$ be the torus bundle over $B_1$ which is the inverse image $\pi_G^{-1}(B_1)$, then its rigid analytic $T$-dual $U_1^\vee$ can be regarded as the moduli space of $(L_b,\xi_b)$, with $L_b$ a torus fiber of $\pi_G$ and $\xi_b$ a unitary rank 1 local system. Under the isotopy from $L_\mathrm{ref}$ to $L_b$, every loop $\theta_i$ traces out a cylinder $\mathit{\Theta}_i$ with boundary lying inside $L_\mathrm{ref}\cap L_b$, and the loop $\theta_0^{(1)}$ traces out the cylinder $\mathit{\Theta}_0^{(1)}$. Identifying $H_1(L_b,\mathbb{Z})$ with $H_1(L_\mathrm{ref},\mathbb{Z})$, $\xi_b$ is determined by its holonomies along $\theta_1,\cdot\cdot\cdot,\theta_{n-1},\theta_0^{(1)}$; while $L_b$ is specified by the symplectic areas of $\mathit{\Theta}_1,\cdot\cdot\cdot,\mathit{\Theta}_{n-1},\mathit{\Theta}_0^{(1)}$. Up to a multiplicative constant, this provides a set of coordinates on $U_1^\vee\subset(\mathbb{K}^\ast)^n$:
\begin{equation}\label{eq:mircoo}
(x_1,\cdot\cdot\cdot,x_{n-1},z)=\left(T^{\int_{\mathit{\Theta}_1}\omega_{\overline{X}}}\xi_b(\theta_1),\cdot\cdot\cdot,T^{\int_{\mathit{\Theta}_{n-1}}\omega_{\overline{X}}}\xi_b(\theta_{n-1}),T^{\int_{\mathit{\Theta}_0^{(1)}}\omega_{\overline{X}}}\xi_b\left(\theta_0^{(1)}\right)\right),
\end{equation}
where $T$ is the Novikov parameter corresponding to the symplectic form $\omega_{\overline{X}}$.
For every Lagrangian brane $(L_b,\xi_b)$ with $L_b\subset U_1$, its obstruction is given by (Section 2.2 of $\cite{aak}$)
\begin{equation}\label{eq:pot}\mathfrak{m}_0(L_b,\xi_b)=\sum_{\beta\in\pi_2(X,L_b)\setminus0}T^{\int_\beta\omega_{\overline{X}}}\xi_b(\partial\beta)\mathrm{ev}_\ast\big[\mathcal{M}_1(L_b,\beta)\big],\end{equation}
where $\mathcal{M}_1(L_b,\beta)$ is the moduli space of holomorphic discs with one boundary marked point representing the relative homotopy class $\beta$ and $\mathrm{ev}:\mathcal{M}_1(L_b,\beta)\rightarrow L_b$ is the evaluation map. The pair $(L_b,\xi_b)$ defines a \textit{weakly unobstructed} object in the extended Fukaya category $\widetilde{\mathcal{F}uk}(\overline{X})$ (see Appendix \ref{section:localized}), i.e. it satisfies the weak Maurer-Cartan equation
\begin{equation}\mathfrak{m}_0(L_b,\xi_b)=W^\vee(L_b,\xi_b)e_{L_b},\end{equation}
where $e_{L_b}$ is the unit of $H^0(L_b,\mathbb{K})$, and $W^\vee(L_b,\xi_b)$ is a regular function on $U_1^\vee$ defined by
\begin{equation}\label{eq:weicon}W^\vee(L_b,\xi_b)=\sum_{\beta\in\pi_2(\overline{X},L_b),\beta\cdot D=1}n(L_b,\beta)T^{\int_\beta\omega_{\overline{X}}}\xi_b(\partial\beta).\end{equation}

Recall from Section \ref{section:gghl} that we will use $A$ to denote the set of toric divisors in $\overline{X}$. Using the Lagrangian isotopy between $\mu_{\overline{X}}^{-1}(\bullet)$ and $L_b$, together with the result of $\cite{co}$ we have:
\begin{lemma}[$\cite{cll}$, Proposition 4.30]
For $L_b\subset U_1$ a regular fiber of $\pi_G$, regarded as a Lagrangian submanifold of $\overline{X}$, the algebraic counts of Maslov index 2 stable holomorphic discs $n(L_b,\beta)\neq0$ only when $\beta=\beta_\alpha+\gamma$, where $\beta_\alpha$ are basic disc classes corresponding to $\alpha\in A$, $\gamma\in H_2(\overline{X},\mathbb{Z})$ are classes represented by rational curves. Moreover $n(L_b,\beta_\alpha)=1$ for all $\alpha\in A$.
\end{lemma}
From $(\ref{eq:mircoo})$ and the above lemma we deduce
\begin{lemma}[\cite{aak}, Lemma 8.2]
In the chart $U_1^\vee$, the superpotential $W^\vee$ is given by
\begin{equation}W^\vee=\sum_{\alpha\in A}\left(1+\sum_{\gamma\in H_2(\overline{X},\mathbb{Z})}n(L_b,\beta_\alpha+\gamma)T^{\int_\gamma\omega_{\overline{X}}}\right)T^{\rho(\alpha)}\mathbf{x}^\alpha z^{-1},\end{equation}
where $\mathbf{x}^\alpha=x_1^{\alpha_1}\cdot\cdot\cdot x_{n-1}^{\alpha_{n-1}}$ and $\alpha_i$ is the $i$-th entry of $\alpha\in\mathbb{Z}^{n-1}$.
\end{lemma}
Similar discussions can be carried out on the other chamber $B_2$. Let $L_\mathrm{ref}$ with $b\in B_2$ be the reference fiber, choose a basis $\theta_1,\cdot\cdot\cdot,\theta_{n-1},\theta_0^{(2)}$ of $H_1(L_\mathrm{ref},\mathbb{Z})$, where $-\theta_1,\cdot\cdot\cdot,-\theta_n$ correspond to the orbits of the Hamiltonian $T_\mathbb{R}(N_\nu)$-action, and $\theta_0^{(2)}$ is the boundary of a section of the fibration $w_0:\overline{X}\rightarrow\mathbb{C}$ over the disc of radius $b$ centered at 0. Denote by $\mathit{\Theta}_0^{(2)}$ the relative homotopy class traced out by $\theta_0^{(2)}$ under the isotopy between $L_\mathrm{ref}$ and $L_b$, and $\mathit{\Theta}_1,\cdot\cdot\cdot,\mathit{\Theta}_{n-1}$ be same as above, up to a multiplicative constant, the coordinates on $U_2^\vee$ are given by
\begin{equation}\label{eq:mircoo2}
(x_1',\cdot\cdot\cdot,x_{n-1}',y)=\left(T^{\int_{\mathit{\Theta}_1}\omega_{\overline{X}}}\xi_b(\theta_1),\cdot\cdot\cdot,T^{\int_{\mathit{\Theta}_{n-1}}\omega_{\overline{X}}}\xi_b(\theta_{n-1}),T^{\int_{\mathit{\Theta}_0^{(2)}}\omega_{\overline{X}}}\xi_b\left(\theta_0^{(2)}\right)\right).
\end{equation}
At this stage we appeal to the following result in $\cite{co}$:
\begin{lemma}[\cite{cll}, Proposition 4.36]
Let $\beta_0$ be the relative homotopy class representing the section of $w_0:\overline{X}\rightarrow\mathbb{C}$ with boundary $\theta_0^{(2)}$. For a regular fiber $L_b\subset\overline{X}$ of $\pi_G$ in $U_2$, $n(L_b,\beta)=1$ when $\beta=\beta_0$, otherwise $n(L_b,\beta)=0$.
\end{lemma}
This implies the following:
\begin{lemma}[\cite{aak}, Lemma 8.3]
On the chart $U_2^\vee$, the superpotential $W^\vee$ admits the expression
\begin{equation}W^\vee=y.\end{equation}
\end{lemma}
The gluing formula of the charts $U_1^\vee$ and $U_2^\vee$ can then be determined by the requirement that $W^\vee$ should define a global regular function on the mirror $X^\vee$. After a completion process $\cite{da2}$, this completes the mirror construction for any toric Calabi-Yau manifold $X$.
\begin{theorem}[Chan-Lau-Leung $\cite{cll}$]
The affine conic bundle $X^\vee$ defined by $(\ref{eq:conic})$ is SYZ mirror to the open Calabi-Yau manifold $X$, with
\begin{equation}g(\mathbf{x})=\sum_{\alpha\in A}\left(1+\sum_{\gamma\in H_2(\overline{X},\mathbb{Z})}n(L_b,\beta_\alpha+\gamma)T^{\int_\gamma{\omega_{\overline{X}}}}\right)T^{\rho(\alpha)}\mathbf{x}^\alpha.\end{equation}
Moreover, the Landau-Ginzburg model $(X^\vee,W^\vee)$ is SYZ mirror to the toric Calabi-Yau manifold $\overline{X}$.
\end{theorem}
To write down the defining equation of $X^\vee$ explicitly, one needs to do non-trivial computations of the algebraic counts of stable discs $n(L_b,\beta_\alpha+\gamma)$. This can be done by relating $n(L_b,\beta_\alpha+\gamma)$ to certain Gromov-Witten invariants, see $\cite{cll,llw1,llw2,cclt}$ for details.

\subsection{The converse mirror construction}\label{section:converse}
This section summarizes the work of Abouzaid-Auroux-Katzarkov $\cite{aak}$ on SYZ mirror symmetry of blow-ups of toric varieties. Similar mirror constructions as in the last section can be done in the converse direction, from the affine conic bundle $X^\vee$ to the rigid analytic Calabi-Yau manifold $X$, using the Lagrangian fibration $\pi_A$ defined in Section $\ref{section:lag-blp}$ as the SYZ fibration $\cite{aak}$. However, in this case we need to make the following assumptions to exclude higher order instanton corrections.
\begin{assumption}\label{assumption:3.1}
$c_1(V)\cdot C>\max(0,H\cdot C)$ for every rational curve $C\subset V$.
\end{assumption}
\begin{assumption}\label{assumption:3.2}
Every rational curve $C\subset\overline{X}^\vee$ satisfies $D^\vee\cdot C\geq0$.
\end{assumption}
Observe that Assumption \ref{assumption:3.1} implies Assumption \ref{assumption:3.2} and every smooth affine toric variety satisfies the assumptions above.\\
Analogous to Lemma \ref{lemma:obstruction}, we have the following result.
\begin{lemma}[\cite{aak}, Proposition 5.1]\label{lemma:obstr2}
The regular fibers of $\pi_A:X^\vee\rightarrow B^\vee$ which bound non-constant holomorphic discs in $X^\vee$ are precisely those having non-trivial intersections with $p^{-1}(H\times\mathbb{C})$.
\end{lemma}
Define the \textit{wall} $\mathcal{W}_\blacklozenge\subset B^\vee$ to be the set of points over which the fiber of $\pi_A$ intersects $p^{-1}(U_H\times\mathbb{C})$ non-trivially, where $U_H$ is the neighborhood appeared in Assumption \ref{assumption:symnei}. Note that when $\dim_\mathbb{C}(X^\vee)=2$, $\mathcal{W}_\blacklozenge$ consists of finitely many open intervals in $B^\vee$ which are parallel to each other. However, when $\dim_\mathbb{C}(X^\vee)\geq3$, $\mathcal{W}_\blacklozenge=\Delta\times\mathbb{R}_{>-\varepsilon}$, which is diffeomorphic to $\mathit{\Pi}_\tau\times\mathbb{R}$, in particular $\dim_\mathbb{R}(\mathcal{W}_\blacklozenge)=\dim_\mathbb{R}(B^\vee)$.
By Lemma \ref{lemma:obstr2}, the fibers over $B^\vee\setminus\mathcal{W}_\blacklozenge$ are tautologically unobstructed, while that over $\mathcal{W}_\blacklozenge$ are potentially obstructed.\bigskip

We will denote by $U_\alpha^\vee$ the chart which is a Lagrangian torus bundle over the connected component $B^\vee_\alpha$ of $B^\vee\setminus\mathcal{W}_\blacklozenge$, over which the monomial of weight $\alpha$ dominates all other monomials in the defining equation $(\ref{eq:trohp})$ of $H$.\\
The rigid analytic $T$-dual of $U_\alpha^\vee$ gives a coordinate chart $U_\alpha$ in the mirror manifold $X$ of $X^\vee$. More precisely, fix a reference fiber $L_\mathrm{ref}\subset U_\alpha^\vee$ of $\pi_A$ with $b\in B_\alpha^\vee$. $H_1(L_\mathrm{ref},\mathbb{Z})$ carries a preferred basis $\theta_1,\cdot\cdot\cdot,\theta_{n-1},\theta_0$ consisting of orbits of the various $S^1$ factors. The chart $U_\alpha\subset X$ is the moduli space of pairs $(L_b,\xi_b)$ with $L_b$ a fiber of $\pi_A$ and $\xi_b\in H^1(L_b,U_\mathbb{K})$. Under the isotopy from $L_\mathrm{ref}$ to $L_b$, the loops $\theta_1,\cdot\cdot\cdot,\theta_{n-1}$ trace out the cylinders $\mathit{\Theta}_1,\cdot\cdot\cdot,\mathit{\Theta}_{n-1}$ respectively; and the loop $\lambda$ traces out the cylinder $\mathit{\Theta}_0$. With these data one can write down the coordinates on $U_\alpha$, up to a multiplicative constant we have:
\begin{equation}\label{eq:sfcalp}
(v_{\alpha,1},\cdot\cdot\cdot,v_{\alpha,n-1},w_{\alpha,0})=\left(T^{\int_{\mathit{\Theta}_1}\omega_\varepsilon}\xi_b(\theta_1),\cdot\cdot\cdot,T^{\int_{\mathit{\Theta}_{n-1}}\omega_\varepsilon}\xi_b(\theta_{n-1}),T^{\int_{\mathit{\Theta}_0}\omega_\varepsilon}\xi_b(\theta_0)\right),
\end{equation}
where $T$ is the Novikov paramter corresponding to the symplectic form $\omega_\varepsilon$. As in the case of toric Calabi-Yau manifolds, one can still use the global regularity of the superpotential to glue together the charts $\{U_\alpha\}_{\alpha\in A}$. However, in this case the anticanonical divisor $D^\vee\subset\overline{X}^\vee$ consists of more irreducible components and we need to analyze their wall-crossing phenomena one by one.\\
Given a partial compactification $X_\bullet^\vee$ of $X^\vee$ satisfying Assumption \ref{assumption:3.2}, $(L_b,\xi_b)$ defines a weakly unobstructed object in the extended Fukaya category $\widetilde{\mathcal{F}uk}(X_\bullet^\vee)$, i.e.
\begin{equation}\mathfrak{m}_0(L_b,\xi_b)=W_\bullet(L_b,\xi_b)e_{L_b},\end{equation}
where $W_\bullet(L_b,\xi_b)$ is determined by a weighted count of Maslov index 2 holomorphic discs bounded by $L_b$ in $X_\bullet^\vee$, in the form of $(\ref{eq:weicon})$. In particular, for each $\alpha\in A$, $W_\bullet:U_\alpha\rightarrow\mathbb{K}$ defines a regular function, and these regular functions glue together to a global regular function on $X$.\\

Using the above formalism, one can first show that the last coordinate $w_{\alpha,0}$ is globally defined.
\begin{lemma}[\cite{aak}, Lemma 5.5]\label{lemma:3.7}
Define
\begin{equation}X_w^\vee=X^\vee\cup\widetilde{V}_0\subset\overline{X}^\vee,\end{equation}
then any pair $(L_b,\xi_b)$ with $L_b\subset U_\alpha^\vee$ is a weakly unobstructed object in the extended Fukaya category $\widetilde{\mathcal{F}uk}(X_w^\vee)$, with
\begin{equation}W_w(L_b,\xi_b)=w_{\alpha,0}.\end{equation}
\end{lemma}
Next, we consider monomials in the remaining coordinates
\begin{equation}\mathbf{v}_\alpha=(v_{\alpha,1},\cdot\cdot\cdot,v_{\alpha,n-1})\in(\mathbb{K}^\ast)^{n-1}.\end{equation}
We will see that these coordinates have wall-crossings. Let $\sigma\in\mathbb{Z}^{n-1}$ be a primitive generator of the fan $\Sigma_V$ defining the toric variety $V$, and let $D_\sigma^\mathrm{in}\subset D_V$ be the open stratum of the toric boundary divisor associated to $\sigma$. The following lemma shows that the monomial $\mathbf{v}_\alpha^\sigma=v_{\alpha,1}^{\sigma_1}\cdot\cdot\cdot v_{\alpha,n-1}^{\sigma_{n-1}}$ appears in a weighted count of holomorphic discs in the partial compactification
\begin{equation}X_\sigma^\vee=X^\vee\cup p^{-1}(D_\sigma^{\mathrm{in}}\times\mathbb{C})\subset\overline{X}^\vee.\end{equation}
\begin{lemma}[\cite{aak}, Lemma 5.6]\label{lemma:3.8}
Let $\kappa\in\mathbb{R}$ which satisfies the equation $\langle\sigma,u\rangle+\kappa=0$ of the facet of the moment polytope $\Delta_V$ of $V$ specified by $\sigma$, and let $\alpha_{\mathrm{min}}\in A$ be such that $\langle\sigma,\alpha_{\mathrm{min}}\rangle$ achieves its minimal value. Any pairing $(L_b,\xi_b)$ with $L_b\subset U_\alpha^\vee$ defines a weakly unobstructed object of $\widetilde{\mathcal{F}uk}(X^\vee_\sigma)$, with
\begin{equation}W_\sigma(L_b,\xi_b)=(1+T^{-\varepsilon}w_0)^{\langle\alpha-\alpha_{\mathrm{min}},\sigma\rangle}T^\kappa\mathbf{v}_\alpha^\sigma.\end{equation}
\end{lemma}
In fact, Lemma \ref{lemma:3.8} can be extended to the case of general monomials in the coordinates $\mathbf{v}_\alpha$. For such $\sigma$ one can still associate a partial compactification $X_\sigma^\vee$ of $X^\vee$. Now the main problem is that $X_\sigma^\vee$ does not necessarily admit an embedding into $\overline{X}^\vee$, so the symplectic form $\omega_\varepsilon$ may not extend to $X_\sigma^\vee$, which prevents us from talking about the holomorphic curve theory of $X_\sigma^\vee$. However, this can be remedied by choosing a K\"{a}hler form $\omega_\sigma$ on $X_\sigma^\vee$ which agrees $\omega_\varepsilon$ outside a small neighborhood of the compactifying divisor. Then a regular fiber $L_b$ of $\pi_A$ lying in the region where $\omega_\sigma=\omega_\varepsilon$ is also a Lagrangian submanifold of $X_\sigma^\vee$. This enables us to remove the assumption that $\sigma$ being a primitive generator of the fan $\Sigma_V$ in the statement of Lemma \ref{lemma:3.8}.

The expressions of $W_\sigma:U_\alpha\rightarrow\mathbb{K}$ for different $\sigma\in\mathbb{Z}^{n-1}$ should glue together to give a global regular function on $X$. This is used to determine the coordinate transformations between different charts $U_\alpha$. Consider two adjacent chambers $B_\alpha^\vee$ and $B_\beta^\vee$ separated by part of the wall $\mathcal{W}_\blacklozenge$, i.e. assume that $\alpha,\beta\in A$ are connected by an edge in the polyhedral decomposition $\mathcal{P}$ mentioned in Section $\ref{section:lag-blp}$. In view of Lemma \ref{lemma:3.7} and the strengthened version of Lemma \ref{lemma:3.8} one obtains:
\begin{proposition}[\cite{aak}, Proposition 5.8]
The wall-crossing maps $(\ref{eq:wall-crossing})$ between the coordinate charts $U_\alpha$ and $U_\beta$ preserve the coordinate $w_0$. For the remaining coordinates, we have
\begin{equation}\label{eq:chgl}\mathbf{v}_\alpha^\sigma=(1+T^{-\varepsilon}w_0)^{\langle\beta-\alpha,\sigma\rangle}\mathbf{v}_\beta^\sigma,\end{equation}
for any $\sigma\in\mathbb{Z}^{n-1}$.
\end{proposition}
Up to the completion process of the mirror coordinates mentioned above, this shows that the SYZ mirror manifold of $X^\vee$ is a toric Calabi-Yau manifold $\overline{X}$ with an anticanonical divisor $D$ removed, which we denote by $X$. To determine the mirror Landau-Ginzbug model of $\overline{X}^\vee$, it still remains to compute the superpotential $W:X\rightarrow\mathbb{K}$. As remarked above, this is simply given by taking the sum of $W_\sigma$'s corresponding to the components of $D^\vee$. Finally we get
\begin{equation}W(L_b,\xi_b)=w_{\alpha,0}+\sum_{i=1}^r(1+T^{-\varepsilon}w_0)^{\langle\alpha-\alpha_i,\sigma_i\rangle}T^{\kappa_i}\mathbf{v}_\alpha^{\sigma_i},\end{equation}
where $\kappa_i\in\mathbb{R}$ satisfies the equation $\langle\sigma_i,u\rangle+\kappa_i=0$ for $1\leq i\leq|A|$ and $u$ lies in the facet of $\Delta_V$. $\sigma_1,\cdot\cdot\cdot,\sigma_r$ are primitive integral generators of $\Sigma_V$. $\alpha_i\in A$ are chosen so that $\langle\sigma_i,\alpha_i\rangle$ is minimal. Denote by $w_i$ the expression
\begin{equation}(1+T^{-\varepsilon}w_0)^{\langle\alpha-\alpha_i,\sigma_i\rangle}T^{\kappa_i}\mathbf{v}_\alpha^{\sigma_i}.\end{equation}
\begin{theorem}[Abouzaid-Auroux-Katzarkov, $\cite{aak}$]
Under Assumption \ref{assumption:3.1}, the Landau-Ginzburg model $(X,W)$ is SYZ mirror to $\overline{X}^\vee$, where
\begin{equation}W=w_0+\cdot\cdot\cdot+w_r.\end{equation}
\end{theorem}

\section{Localized mirror functor}\label{section:localized}
We recall here the construction of the localized mirror functor associated to a weakly unobstructed Lagrangian torus in the sense of $\cite{chl}$ and $\cite{chl1}$, and then apply it to the setting of Section \ref{section:fuk}. For simplicity, assume that $T\subset Y$ is a monotone Lagrangian torus in the monotone symplectic manifold $Y$ with $\dim_\mathbb{R}(Y)=2n$. The idea is that although family Floer theory should be regarded as an infinite direct sum of Yoneda modules associated to the fibers (equipped with local systems), in the local case of a chart $U\subset Y$ containing $T$, the family Floer theory associated to the putative SYZ fibration $\pi$ restricted to $U$ can be replaced by the Yoneda module associated to the Lagrangian torus $T$ equipped with $\mathbb{K}^\ast$ local systems. This is actually the cases of toric Calabi-Yau manifolds and their mirrors, where $T$ is represented by a fiber $L_b$ of the corresponding SYZ fibration.\\
Given a weakly unobstructed, Spin Lagrangian submanifold $L\subset Y$, denote by $\mathcal{M}_\mathrm{weak}(L)$ the space of weak bounding cochains of $L$. Denote by $\mathit{\Lambda}_+\subset\mathit{\Lambda}_0$ the maximal ideal in the Novikov ring
\begin{equation}\mathit{\Lambda}_0=\left\{\sum_{i=1}^\infty a_iT^{\lambda_i}\left|a_i\in\mathbb{C},\lambda_i\geq0,\lambda_i\rightarrow\infty\right.\right\}\end{equation}
recall that $b\in C^1(L,\mathit{\Lambda}_+)$ is called a weak bounding cochain if it satisfies the Maurer-Cartan equation
\begin{equation}\sum_{k=0}^\infty\mathfrak{m}_k(b,\cdot\cdot\cdot,b)=W(L,b)\cdot e_L.\end{equation}
The mirror matrix factorization will be constructed using the non-trivial deformation of the $A_\infty$ structure of the monotone Fukaya category $\mathcal{F}uk(Y)$ induced by the weak bounding cochains. Namely we can define an extended Fukaya category $\widetilde{\mathcal{F}uk}(Y)$ whose objects are $L\times\mathcal{M}_\mathrm{weak}(L)$ with $L\in\mathrm{Ob}\big(\mathcal{F}uk(Y)\big)$ and whose morphism spaces are
\begin{equation}\mathit{CF}^\ast\big((L_1,b_1),(L_2,b_2)\big):=\mathit{CF}^\ast(L_1,L_2),\end{equation}
where $b_i\in\mathcal{M}_\mathrm{weak}(L_i)$. More precisely, for $\mathit{CF}^\ast(L_1,L_2)$ to be well-defined, we need to work over a direct summand $\mathcal{F}uk(Y)_\lambda\subset\mathcal{F}uk(Y)$ consisting of Lagrangians with $\mathfrak{m}_0(L)=\lambda\cdot e_L$ instead of the whole Fukaya category. The deformed $A_\infty$ operations on $\widetilde{\mathcal{F}uk}(Y)_\lambda$
\begin{equation}\mathfrak{m}_k^{b_0,\cdot\cdot\cdot,b_k}:\mathit{CF}^\ast(L_0,L_1)[1]\otimes\cdot\cdot\cdot\otimes\mathit{CF}^\ast(L_{k-1},L_k)[1]\rightarrow \mathit{CF}^\ast(L_0,L_k)[1]\end{equation}
are defined to be
\begin{equation}\mathfrak{m}_k^{b_0,\cdot\cdot\cdot,b_k}(x_1,\cdot\cdot\cdot,x_k)=\sum_{l_0,\cdot\cdot\cdot,l_k}\mathfrak{m}_{k+l_0+\cdot\cdot\cdot+l_k}\left(b_0^{\otimes l_0},x_1,b_1^{\otimes l_1},\cdot\cdot\cdot,x_k,b_k^{\otimes l_k}\right).\end{equation}
To get a mirror functor associated to the monotone Lagrangian torus $T\subset Y$ fixed at the beginning of this appendix, we specialize to the case when $L=T$ in the above, so the localized mirror functor is actually defined to be the Yoneda module associated to $(T,b)\in\mathrm{Ob}\big(\mathcal{F}uk(Y)\big)\times\mathcal{M}_\mathrm{weak}(T)$. We emphasize that there is no need to require that $T$ lies in the direct summand $\widetilde{\mathcal{F}uk}(Y)_\lambda$.
\begin{definition}[Cho-Hong-Lau $\cite{chl1}$]
Fix a $\lambda\in\mathit{\Lambda}$, the $A_\infty$ functor
\begin{equation}\widehat{\Phi}_\mathrm{CHL}^{(T,b)}:\widetilde{\mathcal{F}uk}(Y)_\lambda\rightarrow\mathit{MF}\big(W(T,b)-\lambda\big)\end{equation}
is defined as follows:
\begin{itemize}
\item For $(L_1,b_1)\in\mathrm{Ob}\big(\widetilde{\mathcal{F}uk}(Y)_\lambda\big)$,
\begin{equation}\widehat{\Phi}_\mathrm{CHL}^{(T,b),0}(L_1,b_1):=\Big(\mathit{CF}^\ast\big((T,b),(L_1,b_1)\big),\mathfrak{m}_1^{b,b_1}\Big).\end{equation}
\item For $(L_1,b_1),\cdot\cdot\cdot,(L_k,b_k)\in\mathrm{Ob}\big(\widetilde{\mathcal{F}uk}(Y)_\lambda\big)$ and $x_i\in\mathit{CF}^\ast\big((L_i,b_i),(L_{i+1},b_{i+1})\big)$ for $i=1,\cdot\cdot\cdot,k$.
\begin{equation}\widehat{\Phi}_\mathrm{CHL}^{(T,b),k}(x_1,\cdot\cdot\cdot,x_k):\mathit{CF}^\ast\big((T,b),(L_1,b_1)\big)\rightarrow\mathit{CF}^\ast\big((T,b),(L_k,b_k)\big)\end{equation}
is defined as
\begin{equation}\widehat{\Phi}_\mathrm{CHL}^{(T,b),k}(x_1,\cdot\cdot\cdot,x_k)(y)=(-1)^{\left(\sum_i\deg x_i+k\right)(\deg y+1)}\mathfrak{m}_{k+1}^{b,b_0,\cdot\cdot\cdot,b_k}(y,x_1,\cdot\cdot\cdot,x_k),\end{equation}
where $y\in\mathit{CF}^\ast\big((T,b),(L_1,b_1)\big)$.
\end{itemize}
\end{definition}
In $\cite{chl}$, the definition of the above $A_\infty$ functor is formulated in the language of Lagrangian submanifolds equipped with flat line bundles, and the definition is given in the special case when $L_1,\cdot\cdot\cdot,L_k$ are monotone and $b_1=\cdot\cdot\cdot=b_k=0$, or equivalently, in the case when the line bundles $\mathcal{L}_i\rightarrow L_i$ are equipped with trivial connections. $\widehat{\Phi}_\mathrm{CHL}^{(T,b)}$ defined above then becomes an $A_\infty$ functor defined on $\mathcal{F}uk(Y)_\lambda$. Under this formulation, the superpotentials $W(T,b)$ and $W(L_i)$ are naturally defined on $(\mathit{\Lambda}^\ast)^n$ rather than $\mathit{\Lambda}^n$.

We further restrict ourselves to the case when $Y=\big(\mathit{Bl}_K(\mathbb{C}^2),\Omega_1\big)$, which is the situation of Section \ref{section:fuk} where the above formalism is applied to a monotone Lagrangian torus $\widetilde{T}_i\subset \mathit{Bl}_K(\mathbb{C}^2)$. Since $\widetilde{T}_i$ is monotone, we can choose $b_i\in\mathcal{M}_\mathrm{weak}(\widetilde{T}_i)$ so that the superpotential $W\big(\widetilde{T}_i,b_i\big)=0$. With this choice, $\widehat{\Phi}_\mathrm{CHL}^{(\widetilde{T}_i,b_i)}$ defines an $A_\infty$ functor
\begin{equation}\widehat{\Phi}_\mathrm{CHL}^{(\widetilde{T}_i,b_i)}:\mathcal{F}uk\big(\mathit{Bl}_K(\mathbb{C}^2)\big)_{\widetilde{\lambda}}\rightarrow\mathit{MF}\big(W(\widetilde{T}_i)\big).\end{equation}
Namely as objects of the non-zero eigensummand $\mathcal{F}uk\big(\mathit{Bl}_K(\mathbb{C}^2)\big)_{\widetilde{\lambda}}$, we equip the Lagrangian tori $\widetilde{T}_i$'s with trivial bounding cochains. Since the Lagrangian tori $\widetilde{T}_1,\cdot\cdot\cdot,\widetilde{T}_p$ are Floer theoretically orthogonal, applying the above functor to each $\widetilde{T}_i$ we get an $A_\infty$ functor
\begin{equation}\widehat{\Phi}_\mathrm{CHL}:\mathcal{F}uk\big(\mathit{Bl}_K(\mathbb{C}^2)\big)_{\widetilde{\lambda}}\rightarrow\mathit{MF}(X,W),\end{equation}
which is exactly the functor appeared in $(\ref{eq:localmf})$.

\end{document}